\documentclass[11pt,oneside,reqno]{amsart}
\usepackage{amssymb}
\usepackage{amsmath}
\usepackage{amscd}
\usepackage{amsfonts}
\usepackage{mathrsfs}
\usepackage{stmaryrd}
\usepackage{tikz}
\usetikzlibrary{arrows}
\usepackage[T1]{fontenc}
\allowdisplaybreaks

\makeatletter
\DeclareFontFamily{U}{tipa}{}
\DeclareFontShape{U}{tipa}{m}{n}{<->tipa10}{}
\newcommand{\arc@char}{{\usefont{U}{tipa}{m}{n}\symbol{62}}}%

\newcommand{\arc}[1]{\mathpalette\arc@arc{#1}}

\newcommand{\arc@arc}[2]{%
  \sbox0{$\m@th#1#2$}%
  \vbox{
    \hbox{\resizebox{\wd0}{\height}{\arc@char}}
    \nointerlineskip
    \box0
  }%
}
\makeatother

\newcommand{\nord}{\mbox{\scriptsize ${\circ\atop\circ}$}}

\usepackage{bbm}

\theoremstyle{definition}
\newtheorem{theorem}{Theorem}[section]

\newtheorem{thm}[theorem]{Theorem}
\newtheorem{prop}[theorem]{Proposition}

\newtheorem{defn}[theorem]{Definition}
\newtheorem{lemma}[theorem]{Lemma}
\newtheorem{cor}[theorem]{Corollary}
\newtheorem{prop-def}{Proposition-Definition}[section]

\newtheorem{rema}[theorem]{Remark}

\newtheorem{exam}[theorem]{Example}
\newtheorem{nota}[theorem]{Notation}

\newcommand{\R}{{\mathbb R}}
\newcommand{\N}{{\mathbb N}}
\newcommand{\C}{{\mathbb C}}
\newcommand{\Z}{{\mathbb Z}}

\newcommand{\one}{\mathbf{1}}
\renewcommand{\d}{\mathbf{d}}
\newcommand{\wt}{\mbox{\rm wt}\ }

\newcommand{\Hol}{\text{Hol}}
\newcommand{\state}[1]{{\left| #1\right\rangle}}

\begin{document}

\setlength{\oddsidemargin}{0cm} \setlength{\evensidemargin}{0cm}
\baselineskip=18pt

\title[MOSVA and modules over 2d orientable space forms]{Meromorphic open-string vertex algebras and modules over two-dimensional orientable space forms}
\author{Fei Qi}

\maketitle

\begin{abstract}
    We study the meromorphic open-string vertex algebras and their modules over the two-dimensional Riemannian manifolds that are complete, connected, orientable, and of constant sectional curvature $K\neq 0$. Using the parallel tensors, we explicitly determine a basis for the meromorphic open-string vertex algebra, its module generated by an eigenfunction for the Laplace-Beltrami operator, and its irreducible quotient. We also study the modules generated by the lowest weight subspace satisfying a geometrically interesting condition. It is showed that every irreducible module of this type is generated by some (local) eigenfunction on the manifold. A classification is given for modules of this type admitting a composition series of finite length. In particular and remarkably, if every composition factor is generated by eigenfunctions of special eigenvalue $p(p-1)K$ for some $p\in \Z_+$, then the module is completely reducible. 
\end{abstract}

\section{Introduction}

Vertex algebras are algebraic structures formed by vertex operators satisfying commutativity and associativity. In mathematics, they arose naturally in the study of representations of infinite-dimensional Lie algebras and the Monster group (see  \cite{B} and \cite{FLM}). In physics, they are used in the study of two-dimensional conformal field theory (see \cite{BPZ} and \cite{MS}). The commutativity and associativity allow us to view vertex algebras as analogs to the commutative associative algebras. 


Meromorphic open-string vertex algebras (MOSVAs hereafter) are algebraic structures formed by meromorphic vertex operators satisfying only associativity without commutativity. They were introduced by Huang in 2012 (see \cite{H-MOSVA}), as special cases of the open-string vertex algebras introduced by Huang and Kong in 2003 (see \cite{HK-OSVA}) where all correlation functions are rational functions. Similar to the case of vertex algebras, MOSVAs can be viewed as analogues of associative algebras that are not necessarily commutative. 

To exhibit a natural and nontrivial example, Huang introduced the MOSVA associated with a Riemannian manifold in 2012 (see \cite{H-MOSVA-Riemann}), using the parallel sections of the tensor algebra of the affinized tangent bundle. Given a complex-valued smooth function $f$ on an open subset $U$ of the manifold, Huang also constructed the module generated by $f$, and proved that association of $U$ with the sum of modules generated by all smooth functions over $U$ gives a presheaf of modules for the MOSVA. 

Of particular interest are the modules generated by eigenfunctions of the Laplace-Beltrami operator. As such functions can be understood as quantum states in quantum mechanics, the modules they generate can be understood as the string-theoretic excitement to the quantum states. It is Huang's idea that the modules for the MOSVA generated by the eigenfunctions and the yet-to-be-defined intertwining operators among these modules may lead to a mathematical construction of the quantum two-dimensional nonlinear $\sigma$-model. 

This paper studies the example of MOSVAs and modules over two-dimensional orientable space forms, i.e., two-dimensional Riemannian manifolds that are complete, connected, orientable, and with constant curvature $K\neq 0$. The organization and results are as follows,

In Section 2, we briefly review the previously known results in [H1] and [H2]. The axioms
for the MOSVA and modules are slightly modified using the results in [Q1] to make it easier to
verify.

In Section 3, we determine a basis for the MOSVA. We study the holonomy groups of all the bundles involved in Huang's construction. Then we write down the parallel sections, use them to determine a basis explicitly for the MOSVA, and compute its graded dimension. It turns out that the graded dimension is related to the hypergeometric function $_2F_1$. 

In Section 4, we focus on the modules generated by the eigenfunctions for the Laplace-Beltrami operator. We first prove a lemma on higher-order covariant derivatives, which is then used to show that the covariant derivative of an eigenfunction along every parallel tensor is a scalar multiple of the eigenfunction. The scalar turns out to be a polynomial function in terms of the eigenvalue with roots sitting in $\{p(p-1)K: p\in \Z_+\}$. Using these results, we determine a basis explicitly for every module generated by an eigenfunction. We also show that two such modules are isomorphic if and only if they are generated by eigenfunctions with the same eigenvalues over space forms of the same sectional curvature.

In Section 5, we study the irreducible modules generated by the eigenfunctions, which are quotients of the modules constructed in Section 4 by their unique maximal submodules. To identify these quotients, we prove a formula on the lowest weight projection of $Y(v,x)w$ for homogeneous $v\in V$ and $w\in W$, which plays a crucial role in the subsequent discussions. Then we take two quotients consecutively to the module identified in Section 4. The resulted quotient is irreducible if the eigenfunction is of generic eigenvalues $\lambda \notin \{p(p-1)K: p\in \Z_+\}$. On the other hand, if we start with an eigenfunctions of special eigenvalues $p(p-1)K$ for some $p\in \Z_+$, an additional quotient is needed to obtain an irreducible module, which is remarkably different to those with generic eigenvalues. 

In Section 6, we study the lowest weight modules in general, which are modules generated by its lowest weight subspace. We show that there exists a bijective correspondence between irreducible lowest weight modules for the MOSVA and irreducible modules for the algebra of parallel tensors. We then focus on lowest weight modules whose lowest weight subspace, as a module for the algebra of parallel tensors, satisfies a geometrically interesting condition, called the covariant derivative condition. It turns out that if such a module is irreducible, then it is isomorphic to one identified in Section 5. In other words, every irreducible module of this type is generated by some (local) eigenfunction on the manifold. We also give a classification of modules of this type admitting a composition series of finite length. Remarkably, if every composition factor is generated by eigenfunctions with special eigenvalues $p(p-1)K$ for some $p\in \Z_+$, then the module is completely reducible.

\textbf{Acknowledgements.} The author would like to thank Yi-Zhi Huang for his long-term support and patient guidance. The author would also like to thank Robert Bryant for his guidance on holonomy groups over the tensor powers of the tangent bundle and his correction on the statement of Lemma \ref{HolTensor}. The author would also like to thank Igor Frenkel, Nicholas Lai, Eric Schippers, and Nolan Wallach for helpful discussion. 

\section{Previously known results}

In this section, we recall some background knowledge in \cite{H-MOSVA}, \cite{H-MOSVA-Riemann} and \cite{Q-Mod}. 

\subsection{Axioms of the meromorphic open-string vertex algebra and its left module}

\begin{defn}\label{DefMOSVA}
{\rm A {\it meromorphic open-string vertex algebra} (hereafter MOSVA) is a $\Z$-graded vector space 
$V=\coprod_{n\in\Z} V_{(n)}$ (graded by {\it weights}) equipped with a {\it vertex operator map}
\begin{eqnarray*}
   Y_V:  V\otimes V &\to & V[[x,x^{-1}]]\\
	u\otimes v &\mapsto& Y_V(u,x)v,
  \end{eqnarray*}
and a {\it vacuum} $\one\in V$, satisfying the following axioms:
\begin{enumerate}
\item Axioms for the grading:
\begin{enumerate}
\item {\it Lower bound condition}: When $n$ is sufficiently negative,
$V_{(n)}=0$.
\item {\it $\d$-commutator formula}: Let $\d_{V}: V\to V$
be defined by $\d_{V}v=nv$ for $v\in V_{(n)}$. Then for every $v\in V$
$$[\d_{V}, Y_{V}(v, x)]=x\frac{d}{dx}Y_{V}(v, x)+Y_{V}(\d_{V}v, x).$$
\end{enumerate}

\item Axioms for the vacuum: 
\begin{enumerate}
\item {\it Identity property}: Let $1_{V}$ be the identity operator on $V$. Then
$Y_{V}(\mathbf{1}, x)=1_{V}$. 
\item {\it Creation property}: For $u\in V$, $Y_{V}(u, x)\mathbf{1}\in V[[x]]$ and 
$\lim_{x\to 0}Y_{V}(u, x)\mathbf{1}=u$.
\end{enumerate}

\item {\it $D$-derivative property and $D$-commutator formula}:
Let $D_V: V\to V$ be the operator
given by
$$D_{V}v=\lim_{x\to 0}\frac{d}{dx}Y_{V}(v, x)\one$$
for $v\in V$. Then for $v\in V$,
$$\frac{d}{dx}Y_{V}(v, x)=Y_{V}(D_{V}v, x)=[D_{V}, Y_{V}(v, x)].$$

\item {\it Weak associativity with pole-order condition}: For every $u_1, v\in V$, there exists $p\in \mathbb{N}$ such that for every $u_2\in V$, 
$$(x_0+x_2)^p Y_V(u_1, x_0+x_2)Y_V(u_2, x_2)v = (x_0+x_2)^p Y_V(Y_V(u_1, x_0)u_2, x_2)v.$$

\end{enumerate}  }
\end{defn}

\begin{rema}
This definition is slightly more special than the definition given by Huang in \cite{H-MOSVA}, where $p$ is not necessarily independent of $u_2$. This independence is called pole-order condition and can be used to simplify the verification of axioms. Please see \cite{Q-Mod} for a detailed discussion. 
\end{rema}

\begin{defn}\label{DefMOSVA-L}
Let $V$ be a meromorphic open-string vertex algebra.
A \textit{left $V$-module} is a $\C$-graded vector space 
$W=\coprod_{m\in \C}W_{[m]}$ (graded by \textit{weights}), equipped with 
a \textit{vertex operator map}
\begin{eqnarray*}
Y_W^L: V\otimes W & \to & W[[x, x^{-1}]]\\
u\otimes w & \mapsto & Y_W^L(u, x)w,
\end{eqnarray*}
an operator $\d_{W}$ of weight $0$ and 
an operator $D_{W}$ of weight $1$, satisfying the 
following axioms:
\begin{enumerate}

\item Axioms for the grading: 
\begin{enumerate}
\item \textit{Lower bound condition}:  When $\text{Re}{(m)}$ is sufficiently negative,
$W_{[m]}=0$. 
\item  \textit{$\mathbf{d}$-grading condition}: for every $w\in W_{[m]}$, $\d_W w = m w$.
\item  \textit{$\mathbf{d}$-commutator formula}: For $u\in V$, 
$$[\mathbf{d}_{W}, Y_W^L(u,x)]= Y_W^L(\mathbf{d}_{V}u,x)+x\frac{d}{dx}Y_W^L(u,x).$$
\end{enumerate}

\item The \textit{identity property}:
$Y_W^L(\one,x)=1_{W}$.

\item The \textit{$D$-derivative property} and the  \textit{$D$-commutator formula}: 
For $u\in V$,
\begin{eqnarray*}
\frac{d}{dx}Y_W^L(u, x)
&=&Y_W^L(D_{V}u, x) \\
&=&[D_{W}, Y_W^L(u, x)].
\end{eqnarray*}

\item {\it Weak associativity with pole-order condition}: For every $v_1\in V, w\in W$, there exists $p\in \mathbb{N}$ such that for every $v_2\in V$, 
$$(x_0+x_2)^p Y_W^L(v_1, x_0+x_2)Y_W^L(v_2, x_2)w = (x_0+x_2)^p Y_W^L(Y_V(v_1, x_0)v_2, x_2)w. $$
\end{enumerate} 
\end{defn}


\subsection{Example: Noncommutative Heisenberg}

The first nontrivial example of MOSVA is constructed by Huang in \cite{H-MOSVA}. We should recall the construction here.

Let $\mathfrak{h}$ be a finite-dimensional Euclidean space over $\R$. We define a vector space  
$$\hat{\mathfrak{h}} = \mathfrak{h} \otimes_\R \C[t, t^{-1}] \oplus \C \mathbf{k},$$ 
which is the ambient vector space of the Heisenberg Lie algebra. Note that 
$$\hat{\mathfrak{h}} = \hat{\mathfrak{h}}_- \oplus \hat{\mathfrak{h}}_0 \oplus \hat{\mathfrak{h}}_+,$$
where 
\begin{align*}
    \hat{\mathfrak{h}}_+ &= \mathfrak{h}\otimes_\R t\C[t],\\
    \hat{\mathfrak{h}}_0 &= \mathfrak{h}\otimes_\R \C \oplus \C\mathbf{k},\\
    \hat{\mathfrak{h}}_+ &= \mathfrak{h}\otimes_\R t^{-1}\C[t^{-1}].
\end{align*}

Let $N(\hat{\mathfrak{h}})$ be the quotient of the tensor algebra $T(\hat{h})$ of $\hat{h}$ modulo the two-sided ideal $I$ generated by
\begin{eqnarray}
&(a\otimes t^{m})\otimes (b\otimes t^{n})
- (b\otimes t^{n})\otimes (a\otimes t^{m})
-m(a, b)\delta_{m+n, 0}\mathbf{k},& \nonumber\\
&(a\otimes t^{k})\otimes (b\otimes t^{0})
-(b\otimes t^{0})\otimes (a\otimes t^{k}),\nonumber&\\
&(a\otimes t^{k})\otimes \mathbf{k}-\mathbf{k}\otimes (a\otimes t^{k}),& \label{IdealT(h)}
\end{eqnarray}
for $a, b\in \mathfrak{h}$, $m\in \Z_{+}$, $n\in -\Z_{+}$, $k\in \Z\setminus\{0\}$. Note that in the quotient, there are no relations between $X \otimes t^m $ and $Y\otimes t^n$ for $m, n\in Z_+$, for $m, n\in \Z_-$, and for $m=n=0$. This is the main difference to the usual construction of Bosonic Fock space, where $$(a\otimes t^m)\otimes (b\otimes t^n)-(b\otimes t^n)\otimes (a\otimes t^m)$$ 
is also included in the generators of $I$, for each $a, b\in \mathfrak{h}, m, n\in \Z_{\pm}$ and $m=n=0$. Nevertheless, the PBW structure still holds for this quotient: $N(\hat{\mathfrak{h}}) \simeq T(\hat{\mathfrak{h}}_-)  \otimes T(\hat{\mathfrak{h}}_+) \otimes T(\mathfrak{h})\otimes T(\C \mathbf{k})$ as vector spaces (see \cite{H-MOSVA}, Proposition 3.1) 

Let $\C = \C \mathbf{1}$ be a one-dimensional vector space on which $\mathfrak{h}$ acts by 0. Define the action of $\mathbf{k}$ by 1 and $\hat{\mathfrak{h}}_+$ by 0. One can prove that  the induced module $N(\hat{\mathfrak{h}}) \otimes_{N(\hat{\mathfrak{h}}_+\oplus\hat{\mathfrak{h}}_0)} \C$ is isomorphic to $T(\hat{\mathfrak{h}}_-)$ as a vector space. We regard $T(\hat{\mathfrak{h}}_-)$ now as an $N(\hat{\mathfrak{h}})$-module and denote the action of $h\otimes t^j$ by $h(j)$. Then $T(\hat{\mathfrak{h}}_-)$ is spanned by $h_1(-m_1) \cdots h_k(-m_k)\mathbf{1}$ for $k \in \mathbb{N}, h_1, ..., h_k \in \mathfrak{h}, m_1, ..., m_k \in \Z_+$. 

Huang proved the following theorem in \cite{H-MOSVA}. 

\begin{thm}[\cite{H-MOSVA}, Theorem 5.1]\label{NoncomHeis} The left $N(\hat{\mathfrak{h}})$-module $T(\hat{\mathfrak{h}}_-)$ forms a grading-restricted MOSVA with the following vertex operator action: 
\begin{align*}
& Y(h_1(-m_1) \cdots h_k(-m_k)1, x) \\
& \qquad= \nord \frac 1 {(m_1-1)!} \frac {d^{m_1-1}}{dx^{m_1-1}}h_1(x) \cdots \frac 1 {(m_k-1)!} \frac {d^{m_k-1}}{dx^{m_k-1}}h_k(x) \nord,    
\end{align*}
where $h_i(x) = \sum_{n\in \Z} h_i(n)x^{-n-1}$. The normal ordering is defined as follows:
$$\nord h_1(m_1)\cdots h_k(m_k)\nord = h_{\sigma(1)}(m_{\sigma(1)}) \cdots h_{\sigma(k)}(m_{\sigma(k)}),$$
where $\sigma\in S_k$ is the unique permutation such that
\begin{eqnarray}\sigma(1) < \cdots < \sigma(\alpha), \sigma(\alpha+1)< \cdots < \sigma(\beta), \sigma(\beta)<\cdots < \sigma(k),\quad\, \nonumber\\
m_{\sigma(1)}, ...,  m_{\sigma(\alpha)}< 0, m_{\sigma(\alpha+1)}, ..., m_{\sigma(\beta)} >0, m_{\sigma(\beta+1)}, ..., m_{\sigma(k)}= 0. \label{Shuffle}
\end{eqnarray}
\end{thm}


\subsection{MOSVA of a Riemannian manifold} In \cite{H-MOSVA-Riemann}, Huang further constructed a MOSVA for any Riemannian manifold $M$, using the parallel sections of a certain bundle. We recall the construction here. 

Let $M$ be a Riemannian manifold. Let $p\in M$. We consider the affinization of tangent bundle 
$$\widehat{TM} = TM \otimes_\R (M \times \C[t, t^{-1}]) \oplus (M\times \C \mathbf{k}),$$
where $M\times \C[t,t^{-1}]$ and $M\times \C\mathbf{k}$ are trivial bundles over $M$. The fiber of this bundle at $p$ is nothing but 
$$\widehat{T_pM} = T_pM \otimes_\R \C[t, t^{-1}]\oplus \C \mathbf{k}.$$
Analogous to the constructions in the previous subsection, we have 
\begin{align*}
    \widehat{TM} & = \widehat{TM}_+ \oplus \widehat{TM}_0 \oplus \widehat{TM}_-,
\end{align*}
with 
\begin{align*}
    \widehat{TM}_\pm &= TM\otimes_\R (M \times t^{\pm 1}\C[t^{\pm 1}]), \\
    \widehat{TM}_0 &= TM \otimes_\R (M \times \C t^0) \oplus TM \otimes_\R (M \times \C\mathbf{k}).
\end{align*}

Now we look into the tensor algebra bundle $T(\widehat{TM})$. We similarly construct a bundle $N(\widehat{TM})$, whose fiber at each point is obtained by taking the quotient of $T(T_pM)$ versus the two-sided ideal generated by elements in (\ref{IdealT(h)}), for $a, b \in T_pM$. Likewise, $N(\widehat{TM})$ is isomorphic to $T(\widehat{TM}_-) \otimes  T(\widehat{TM}_+)\otimes T(\widehat{TM})\otimes T(\C\mathbf{k})$. 

We now consider the space $\Pi(T(\widehat{TM}_-))$ consisting of parallel sections of the tensor algebra bundle $T(\widehat{TM}_-)$ of $\widehat{TM}_-$. It is well known that $\Pi(T(\widehat{TM}_-))$, as a vector space, is isomorphic to subspace of fixed points $T(\widehat{T_pM}_-)^{\Hol(T(\widehat{TM}_-)}$, where Hol means the holonomy group of the bundle. Recall that Theorem \ref{NoncomHeis} endows the space $T(\widehat{T_pM}_-)$ with a MOSVA structure. In \cite{H-MOSVA-Riemann}, Huang proved that the elements of the holonomy group could be realized as automorphisms of the MOSVA, and on the fixed-point subspace of automorphism there is a MOSVA structure. As a consequence: 

\begin{thm}\label{MOSVA-Riemann}
On the space $\Pi(T(\widehat{TM}_-))$ of parallel sections there is a MOSVA structure, which is a subalgebra of the MOSVA of $T(\widehat{T_pM}_-)$ in Theorem \ref{NoncomHeis}. 
\end{thm}
Moreover, by an analogous argument to the Segal-Sugawara construction, Huang also showed that the Laplace-Beltrami operator on the manifold $M$ is realized as a component of some vertex operator. 


\subsection{Modules generated by smooth functions}\label{module-def}

In \cite{H-MOSVA-Riemann}, Huang considered that the action of the parallel sections $\Pi(T(TM^\C))$ on the space of smooth functions via the covariant derivatives. More precisely, let $U$ be an open subset of $M$; let $f$ be a smooth function on $U$. A parallel section $X \in \Pi((TM^\C)^{\otimes k})$ acts on $f$ by 
$$\psi_U(X)f:=(\sqrt{-1})^k(\nabla^k f)(X).$$ 
Huang proved (Theorem 4.1, \cite{H-MOSVA-Riemann}) that the action respects the associative algebra structure defined by $\otimes$, namely, for $X, Y\in \Pi(T(TM)^\C)$,
$$\psi_U(X\otimes Y) = \psi_U(X)\psi_U(Y).$$
Thus the space of complex-valued smooth functions can be viewed as a module for the associative algebra $\Pi(T(TM)^\C)$. 

Let $p\in U$. Identify $\Pi(T(TM)^\C)$ with $T(T_pM^\C)^{\Hol(T(T_pM^\C))}$, which can be viewed as a subalgebra of $T(T_pM^\C)$. Thus the module $C^\infty(U)$ can be induced, defining  
$$C_p(U) = T(T_pM^\C)\otimes _{T(T_pM^{\C})^{\Hol(T(T_pM^{\C}))}}C^\infty(U).$$
By Theorem 6.5 of \cite{H-MOSVA}, on the vector space $$T(\widehat{T_pM}_-)\otimes C_p(U)$$
there is a natural module structure for the MOSVA $T(\widehat{T_pM}_-)$. In particular, it is a module for the MOSVA $T(\widehat{T_pM}_-)^{\Hol(\widehat{T_pM}_-)} = \Pi(T(\widehat{TM}_-))$. We can thus consider the $\Pi(T(\widehat{TM}_-))$-submodule generated by 
$$1\otimes (1\otimes f)$$
for some $f\in C^\infty(U)$. Since the Laplace-Beltrami operator is a component of some vertex operator, it would be natural to consider the submodule generated by its eigenfunction. As the Laplace-Beltrami operator plays the role of energy operator in quantum mechanics, the submodule can then be interpreted as string-theoretical excitement of the quantum states.


\section{Basis of the MOSVA and modules}

Let $M$ be a two-dimensional Riemannian manifold with constant sectional curvature $K$. For convenience, we assume $M$ is orientable, connected and complete. We will also focus on the case $K\neq 0$. In this section, we will first determine the parallel sections of the tensor algebra of the affinized tangent bundle. Using these parallel sections, we explicitly identify a basis for the MOSVA and modules generated by the eigenfunctions constructed in \cite{H-MOSVA-Riemann}. 


\subsection{Holonomy of the tensor powers of the complexified tangent bundle. }

Recall that the holonomy group of a bundle $E$ based at a point $p\in M$ is the subgroup generated by all the parallel translations along piecewise smooth contractible loops based on $p$. We will simply denote the holonomy group by $\Hol(E)$ since holonomy groups based at different points are isomorphic.   

\begin{lemma}
If $K\neq 0$, then the holonomy group $\Hol(TM)$ of the tangent bundle $TM$ is $SO(2, \R)$. 
\end{lemma}

\begin{proof}
Since $M$ is orientable, we know that $\Hol(TM) \subset SO(2,\R)$ (see \cite{P}). Fix $p, q, r\in M$. Let $\gamma_1, \gamma_2, \gamma_3$ be geodesics connecting $pq, qr$ and $rp$. Let $\alpha_p, \alpha_q, \alpha_r$ be the angles of geodesic triangle $pqr$. Let $v\in T_pM$ be a unit vector. One sees easily that that composition of parallel transport along the concatenation of $\gamma_1, \gamma_2$ and $\gamma_3$ ends up with a unit vector $w\in T_pM$ that is obtained from rotating $v$ by the angle $3\pi - (\alpha_p+\alpha_q+\alpha_r)$, as shown in the following two figures.  

\includegraphics[scale = 0.55]{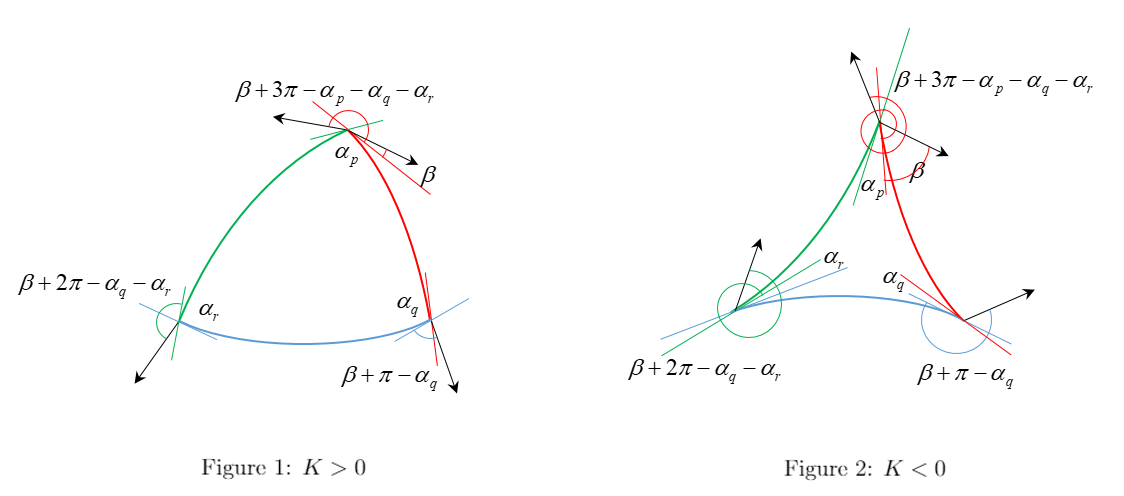}

By Gauss-Bonnet theorem (geodesic triangle version, see \cite{DoCarmo}), 
$$\alpha_p+\alpha_q+\alpha_r = \pi + K\cdot \text{Area}(pqr).$$ 
Now let $q, r$ vary near $p$, so that the area of the geodesic triangle varies continuously within some interval $[0, s]$. Then we see that rotations by angles between $2\pi - Ks$ and $2\pi$ are all included in the holonomy group. These rotations generate all $SO(2, \R)$. 
\end{proof}

We focus on the complexified tangent bundle $\C \otimes_\R TM $, where $\C$ is regarded as a trivial bundle over $M$. 
\begin{nota}
From now on, we shall denote the complexified tangent bundle $\C\otimes_\R TM$ by $E$. We will also omit the $\otimes_\R$ symbol when writing the smooth sections in $\Gamma(E)$. All $\otimes$ symbol will mean $\otimes_\C$ by default unless otherwise stated. 
\end{nota} 
The connection on $E$ is related to that on $TM$ by
$$\nabla (X+iY) = \nabla(X) + \sqrt{-1}\nabla(Y), X, Y \in \Gamma(TM).$$

\begin{lemma}
    $\Hol(E) = \Hol(TM) = SO(2, \R)$. 
\end{lemma}

\begin{proof}
This follows from the observation that the bundle $E$ is essentially the direct sum $TM\oplus \sqrt{-1}\cdot TM$. 
\end{proof}

We will also consider the tensor bundle $E^{\otimes k}$ for each $k\in \Z_+$. 

\begin{lemma} There is a natural surjective homomorphism $\Hol(E) \to \Hol(E^{\otimes k})$ of holonomy groups, where $g\in \Hol(E)$ acts on each fiber $E^{\otimes k}_p$ by 
$$g(v_1\otimes \cdots \otimes v_k) = gv_1\otimes \cdots \otimes gv_k.$$
for any $v_1, ..., v_k\in E_p$. 
\end{lemma}

\begin{proof}
For any piecewise smooth path $\gamma: [0, 1]\to M$ with $\gamma(0)=p$, let $P_{\gamma}: E_p \to E_{\gamma(1)}$ be the parallel transport along $\gamma$ on the bundle $E$; let  $P_{\gamma}^k: E_p^{\otimes k} \to E_{\gamma(1)}^{\otimes k}$ be the parallel transport along $\gamma$ with respect to the bundle $E^{\otimes k}$. Then 
from the definition of the connection on $E^{\otimes k}$: 
$$\nabla(X_1 \otimes \cdots \otimes X_k) = \sum_{i=1}^k X_1 \otimes \cdots \otimes \nabla(X_i) \otimes \cdots \otimes X_k, $$
it follows that
$$P_{\gamma}^k(v_1\otimes \cdots \otimes v_k) = P_{\gamma}v_1 \otimes \cdots \otimes P_{\gamma}v_k.$$
In case $\gamma(t)$ is a loop based at $p$, this essentially realizes every element of $h\in \Hol(E^{\otimes k})$ as $g^{\otimes k}$ for $g\in \Hol(E)$. So the map $g\mapsto g^{\otimes k}$ gives a natural surjective homomorphism  $\Hol(E) \to \Hol(E^{\otimes k})$. 
\end{proof}

\begin{lemma}\label{HolTensor}
The holonomy group of $E^{\otimes k}$ is determined by
$$\Hol(E^{\otimes k}) = \left\{\begin{aligned}
& SO(2, \R) &  &\text{if }k\text{ is odd},\\
& SO(2, \R)/\{\pm 1\} & & \text{if }k\text{ is even}.
\end{aligned}\right.$$
\end{lemma}

\begin{proof}
We analyze the kernel of the homomorphism $SO(2, \R) \to \Hol(E^{\otimes k})$. Fix $p\in M$. For the matrix 
$$M(\alpha)=\begin{bmatrix} \cos\alpha & -\sin\alpha \\
\sin\alpha & \cos\alpha
\end{bmatrix},$$
let $v_{\pm1}\in E_p^\C$ be an eigenvector of $M(\alpha)$ with eigenvalue $e^{ \pm i\alpha}$. The conclusion then follows from 
$$M(\alpha)(v_{i_1}\otimes \cdots \otimes v_{i_k}) = e^{i(m-n)\alpha}(v_{i_1}\otimes \cdots \otimes v_{i_k}),$$
where $m=\#\{j: i_j = 1\}, n = \#\{j: i_j = -1\}$. 
In greater detail, if $k$ is odd then there is no way to make $e^{i(m-n)\alpha}=1$ for every choice of $i_1, ..., i_k$ unless $\alpha =0$; if $k$ is even, then there is no way to make $e^{i(m-n)\alpha}=1$ for every choice of $i_1, ..., i_k$ unless $\alpha = 0$ or $\alpha = \pi$. 
\end{proof}


\subsection{Parallel sections of the tensor powers of the complexified tangent bundle} 

\begin{lemma}
$\Pi(E^{\otimes k}) = 0$ if $k$ is odd. 
\end{lemma}

\begin{proof}
It follows directly from $-1\in\Hol(E^{\otimes k})$. 
\end{proof}

To describe the parallel sections of the even tensor powers of $E$, we introduce the following notations: 
\begin{defn}\label{sections}
Fix $p\in M$, let $\{e_1, e_2\}$ be an orthonormal basis of $T_p M$. Then for some neighborhood $V$ of $p$, we define local sections $X_1, X_2: V \to TV$ by parallel transporting $e_1, e_2$ to every $q\in V$. Finally, we introduce the following local sections of $E$:
$$h_+ = X_1 - \sqrt{-1} X_2, h_- = X_1 + \sqrt{-1} X_2.$$
\end{defn}

\begin{exam}\label{3-8}
Let $M$ be the unit sphere $M = \{(x, y, z)\in \R^3: x^2 + y^2 + z^2 = 1\}$. For $p = (1, 0, 0)$, we take the polar coordinate
$$x = \cos \theta \sin\phi , y =  \sin \theta \sin \phi, z = \cos \phi,$$
where $\theta\in [0, 2\pi), \phi\in (0, \pi)$. Then on the neighborhood $V = M \setminus \{(0, 0, \pm 1)\}$ of $p$, the metric is of the form 
$$ds^2 = d\phi^2 + \sin^2\phi d\theta^2.$$ 
We then take 
$$X_1 = \partial_\phi, X_2 = \frac 1 {\sin \phi}\partial_\theta.$$
In this case, 
$$h_+ = \partial_\phi -  \frac {\sqrt{-1}} {\sin \phi}\partial_\theta, h_- = \partial_\phi +  \frac {\sqrt{-1}} {\sin \phi}\partial_\theta.$$
\end{exam}

\begin{exam}
Let $M$ be a complete hyperbolic surface of genus $g$. From the discussion in \cite{JS}, $M$ is realized as the orbit space $H/ \Gamma$, where $H$ is the Poincar\'e disk
$$H = \{(x,y)\in \R^2: x^2+y^2 < 1\}$$
and $\Gamma$ is a Fuchsian group (with no fixed points on $H$). For $p = (0,0)$, let $V$ be the interior of the fundamental region of $\Gamma$ containing $p$. In other words, $V$ is the interior of a hyperbolic polygon with $2g$ sides. Viewed as a coordinate chart of $M$ near $p$, $V$ can be endowed with a metric of the form
$$ds^2 = \frac {4(dx^2 + dy^2)}{(1-x^2-y^2)^2}. $$
We then take 
$$X_1 = \frac {1-x^2-y^2} 2\partial_x, X_2 = \frac {1-x^2-y^2} 2\partial_y.$$
In this case, 
$$h_+ = \frac {1-x^2-y^2} 2(\partial_x - \sqrt{-1}\partial_y), h_-= \frac {1-x^2-y^2} 2(\partial_x + \sqrt{-1}\partial_y).$$

\end{exam}

\begin{prop} \label{Parallel-Tensors}
Each element in the following set 
$$\left\{h_{i_1} \otimes \cdots \otimes h_{i_k} \left|\begin{aligned} & i_1, ..., i_k \in \{+, -\}, \\
&\#\{j:i_j = +\} = 
\#\{j:i_j = -\}\end{aligned}\right. \right\}$$
extend to a global section $M \to E^{\otimes k}$ and is parallel. The set form a basis of $\Pi(E^{\otimes k})$. 
\end{prop}

\begin{proof}
Fix $p\in M$. Let $M(\alpha)\in SO(2, \R)$ be the matrix as in Lemma \ref{HolTensor}, which acts on $e_1, e_2\in T_pM$ by
\begin{align*}
    M(\alpha)e_1 &= \cos(\alpha) e_1 + \sin(\alpha) e_2,\\ 
    M(\alpha)e_2 &= -\sin(\alpha) e_1 + \cos(\alpha) e_2.
\end{align*}
With this action, $M(\alpha)$ acts on $h_+|_p, h_-|_p$ in the fiber $E_p$ by 
$$M(\alpha) h_+|_p = e^{i\alpha} h_+|_p, M(\alpha)h_-|_p = e^{-i\alpha} h_-|_p. $$
In other words, $(h_\pm)_p$ is an eigenvector of $M(\alpha)$ with eigenvalue $e^{\pm i\alpha}$. 

To determine the parallel sections $\Pi(E^{\otimes k})$, it suffices to study the fixed-point subspace of $(E_p)^{\otimes k}$ under the action of $SO(2, \R)/\{\pm 1\}$. By the same computation shown in Lemma \ref{HolTensor}, we know that 
$$M(\alpha)(h_{i_1}|_p\otimes \cdots \otimes h_{i_k}|_p) = h_{i_1}|_p \otimes \cdots \otimes h_{i_k}|_p$$
if and only if the number of $+$ appearing in $i_1, ..., i_k$ coincides with the number of $-$. Then by simple linear algebra, the set of vectors 
$$\{h_{i_1}|_p\otimes \cdots \otimes h_{i_k}|_p: \#\{j: i_j=+\} = \#\{j: i_j = -\}\}$$
form a basis of the fixed-point subspace $((E_p^\C)^{\otimes k})^{\Hol(E^{\otimes k})}$. 

Since $M$ is complete, every point on $M$ is connected to $p$ by some geodesic path. Every element in this fixed-point subspace $E_p^{\Hol(E^{\otimes k})}$ thus defines a parallel section globally on $M$ via parallel transport along geodesic paths. Moreover, the parallel section defined by $h_{i_1}|_p\otimes \cdots \otimes h_{i_k}|_p$ coincides with the local section $h_{i_1}\otimes \cdots \otimes h_{i_k}$ on $V$, since the latter is also defined via parallel transports. Thus the latter extends to global sections. 
\end{proof}

\subsection{Parallel sections of the tensor algebra of the affinized tangent bundle}
Consider the following affinization of $E$:  
\begin{align*}
    \widehat{E}&= TM\otimes_\R(M\times \C[t, t^{-1}])\oplus (M\times \C k) \\
    &= E\otimes(M\times \C[t, t^{-1}])\oplus (M\times \C k)\\
    &= \widehat{E}_- \oplus \widehat{E}_0 \oplus \widehat{E}_+,
\end{align*}
where
\begin{align*}
    \widehat{E}_\pm &= E \otimes (M \times t^{\pm 1}\C[t^{\pm 1}])\\
    &= \bigoplus_{k=1}^\infty E \otimes (M \times \C t^{\pm k}),\\
    \widehat{E}_0 &= E\otimes (M\times \C t^0) \oplus (M\times \C \mathbf{k}).
\end{align*}
We will need to use the tensor algebra bundle 
$$T(\hat{E}) = (M \times \C) \oplus \hat{E} \oplus \hat{E}^{\otimes 2} \oplus \cdots.$$ 
For each $p\in M$, the fiber $T(\widehat{E})_p$ is simply the tensor algebra of the vector space $\widehat{E_p}$:
$$T(\widehat{E})_p = T(\widehat{E_p})= \C \oplus \widehat{E}_p \oplus (\widehat{E}_p)^{\otimes 2}\oplus \cdots. $$

We construct the bundles $T(\hat{E}_\pm), T(E) $ and $T(M\times \C\mathbf{k})$ similarly. We will rewrite these bundles in the following way:
\begin{align*}
    T(\widehat{E}_\pm) &= \bigoplus_{n=0}^\infty \bigoplus_{k=1}^n\bigoplus_{\substack{m_1+\cdots+  m_k=n\\ m_1, ..., m_k \in \Z_+}} (E \otimes_{\C} (M \times \C t^{\pm m_1})) \otimes_\C \cdots \otimes_\C (E \otimes_\C ({M \times \C t^{\pm m_k}})).\\
    T(E) &= \bigoplus_{n=0}^\infty (E\otimes \C t^0)^{\otimes n} = \bigoplus_{n=0}^\infty E^{\otimes n}. \\
    T(M \times \C \mathbf{k}) &= \bigoplus_{n=0}^\infty (M\times \C \mathbf{k})^{\otimes n} = \bigoplus_{n=0}^\infty (M\times \C \mathbf{k}^{n}).
\end{align*}

We will need to use the parallel sections of all these bundles. Note that essentially $T(\widehat{E}_\pm)$ and $T(E)$ are direct sums of $E^{\otimes n}$.
The following lemma will then apply to determine the parallel section of these bundles. 

\begin{lemma}
Let $B_1, B_2, ...$ be a sequence of vector bundles on $M$. Let $B = \bigoplus_{i=1}^\infty B_i$. Then the parallel sections of $B$ is the direct sum of the parallel sections of $B_i$, i.e., $\Pi(B) = \bigoplus_{i=1}^\infty \Pi(B_i)$.
\end{lemma}

\begin{proof} 
Obviously $\bigoplus_{i=1}^\infty \Pi(B_i) \subseteq \Pi(B)$. We show the inverse inclusion here. Let $X$ be a parallel section of $B$. Fix any $p\in M$ and piecewise smooth loop $\gamma$ based on $p$. Consider $X_p = \sum\limits_{i \text{ finite}} (X_i)_p$, which is a finite sum of components in $(B_1)_p, (B_2)_p$, ... The parallel transport $T_{\gamma(1)}^{B}$ applied on $X$ amounts to the sum of the action of $T_{\gamma(1)}^{B_i}$ on $X_i$. Since it is a direct sum, we necessarily have $T_{\gamma(1)}^{B_i}(X_i)_p = (X_i)_p$. Since $\gamma$ is arbitrarily chosen, we see that $(X_i)_p \in (B_i)_p^{\text{Hol}(B)}$. That is to say, $X_p$ is a finite sum of elements in $(B_i)_p^{\text{Hol}(B)}$. Thus $X$ is a finite sum of parallel sections in $B_i$. So we proved that $\Pi(B)\subseteq \bigoplus_{i=1}^\infty \Pi(B_i). $
\end{proof}

The lemma thus determines the parallel sections of these bundles: 
\begin{prop}\label{ParallelSec}
\begin{align*}
    \Pi(T(\widehat{E}_\pm)) &=\bigoplus_{n=0}^\infty \bigoplus_{k=1}^{n}\bigoplus_{\substack{m_1+\cdots+  m_k=n\\ m_1, ..., m_k \in \Z_+}} \Pi((E \otimes_{\R} (M \times \C t^{\pm m_1})) \otimes_\C \cdots \otimes_\C (E \otimes_\R ({M \times \C t^{\pm m_k}}))\\ 
    &=\bigoplus_{n=0}^\infty \bigoplus_{k=1}^{n}\bigoplus_{\substack{m_1+\cdots+  m_k=n\\ m_1, ..., m_k \in \Z_+}} \text{span} \left\{ (h_{i_1} \otimes t^{\pm m_1}) \otimes \cdots \otimes (h_{i_k} \otimes t^{\pm m_k}): \begin{aligned}  & i_1, ..., i_k \in \{+, -\},\\
    &  \#\{j: i_j=+\}=\#\{j: i_j=-\} \}\end{aligned}\right\}\\
    \Pi(T(E)) &= \bigoplus_{n=0}^\infty \Pi((E\otimes (M\times \C t^0))^{\otimes n}). \\
    &= \bigoplus_{n=0}^\infty \text{span}\left\{(h_{i_1}\otimes t^0 )\otimes \cdots \otimes (h_{i_n}\otimes t^0): \begin{aligned}  & i_1, ..., i_n \in \{+, -\}, \\
    &  \#\{j: i_j=+\}=\#\{j: i_j=-\} \}\end{aligned}\right\}.\\
    \Pi(M \times \C \mathbf{k}) & =  \bigoplus_{n=0}^\infty \Pi(M\times \C \mathbf{k}^n) = \bigoplus_{n=0}^\infty \C \mathbf{k}^n. 
\end{align*}

\begin{proof}
The structure of $\Pi(T(\widehat{E}_\pm))$ and $\Pi(T(E))$ follows directly from the lemma. The structure of $\Pi(M \times \C \mathbf{k})$ follows from the fact that any smooth function $f$ satisfying $\nabla_{\dot\gamma}f = 0$ for every path $\gamma$ is constant. 
\end{proof}

\end{prop}

\subsection{Structure of the MOSVA}\label{3-5}  With the knowledge of the parallel sections, we now explicitly identify the MOSVA constructed by Huang in \cite{H-MOSVA-Riemann}.

\begin{thm} 
Let $V(l, \one)$ be the MOSVA constructed by Huang in \cite{H-MOSVA-Riemann} (cf. Theorem \ref{MOSVA-Riemann}), except that $\mathbf{k}$ acts by a nonzero scalar $l\in \C^{\times}$. Then the vectors
$$h_{i_1}(-m_1)\cdots h_{i_k}(-m_k)\one, i_1, ..., i_k\in \{+, -\}, \#\{j: i_j=+\}=\#\{j: i_j=-\}$$
form a basis for $V(l, \one)$.  Together with the following vertex operator action
\begin{align*}
& Y(h_{i_1}(-m_1) \cdots h_{i_k}(-m_k)\one, x) \\
& \qquad= \nord \frac 1 {(m_1-1)!} \frac {d^{m_1-1}}{dx^{m_1-1}}h_{i_1}(x) \cdots \frac 1 {(m_k-1)!} \frac {d^{m_k-1}}{dx^{m_k-1}}h_{i_k}(x) \nord    
\end{align*}
defines a MOSVA structure on $V(l, \one)$. 
\end{thm}

\begin{proof}
For $l=1$ the theorem has been proved in \cite{H-MOSVA-Riemann}, Proposition 3.3. The generalization to $l\in \C^\times$ is a trivial modification of the whole process. For exposition purposes, we will sketch a direct proof using the computational results in \cite{H-MOSVA} modified by the general central charge.

For convenience, we use $V$ to denote the space $V(l, \one)$. The grading of $V$ is given by specifying $V_n$ to be the span of the vectors $h_{i_1}(-m_1)\cdots h_{i_k}(-m_k)\one$, with $i_1, ..., i_k$ satisfying the conditions in the statement, and $m_1+\cdots + m_k = n$. From Proposition \ref{ParallelSec}, $n\geq 0$. So the grading is lower bounded. 

Let $u = a_1(-m_1) \cdots a_{k_1}(-m_{k_1})\one$ and $v = b_1(-n_1) \cdots b_{k_2}(-n_{k_2})\one$, $a_1, ..., a_{k_1}, b_1, ..., b_{k_2}\in \{h_+, h_-\}, m_1, ..., m_{k_1}, n_1, ..., n_{k_2}\in \Z_+$, the coefficient of each power of $x$ in $Y(u, x)v$ is in $V$. In fact, the coefficients of each power of $x$ is a sum of elements of the form 
$$\nord a_1(p_1)\cdots a_{k_1}(p_{k_1})\nord b_1(-n_1)\cdots b_{k_2}(-n_{k_2})\one.$$
For every such $p_1, ..., p_k$, let $\sigma$ be the unique element in $S_k$ satisfying the condition \begin{eqnarray}
\sigma(1) < \cdots < \sigma(\alpha), & \sigma(\alpha+1)< \cdots< \sigma(\beta), &  \sigma(\beta)<\cdots < \sigma(k_1), \nonumber\\
p_{\sigma(1)}, ...,  p_{\sigma(\alpha)}< 0,& p(\sigma_{\alpha+1}),  ..., p(\sigma_\beta) >0, & p(\sigma_{\beta+1}), ..., p(\sigma_{k_1})= 0,\nonumber \end{eqnarray}
so that the element can be written as
\begin{align*}
    & a_{\sigma(1)}(p_{\sigma(1)}) \cdots a_{\sigma(\alpha)} (p_{\sigma(\alpha)}))
    a_{\sigma(\alpha+1)}(p_{\sigma(\alpha+1)}) \cdots a_{\sigma(\beta)} (p_{\sigma(\beta)}))\\
    & \quad \cdot a_{\sigma(\beta+1)}(0) \cdots a_{\sigma(k_1)} (0))b_1(-n_1)\cdots b_{k_2}(-n_{k_2})\one.
\end{align*}
From the relations of the algebra, for $i,j\in \{+, -\}, n>0$, $h_i(0)$ commutes with all $h_j(-n)$, while $h_i(0)\one = 0$. Thus if there exists some $j$ such that $p_j = 0$, the element is simply zero, which is certainly in $V$. 

Otherwise, if $p_1, ..., p_{k_1}\neq 0$, then $\beta = k_1$. The element in question would then be 
\begin{align*}
    & a_{\sigma(1)}(p_{\sigma(1)}) \cdots a_{\sigma(\alpha)} (p_{\sigma(\alpha)}))
    a_{\sigma(\alpha+1)}(p_{\sigma(\alpha+1)}) \cdots a_{\sigma(\beta)} (p_{\sigma(\beta)})) b_1(-n_1)\cdots b_{k_2}(-n_{k_2})\one.
\end{align*}
From the relation of the algebra, for $i,j\in \{+, -\}$, $p, q>0$, the commutator of $h_i(p)$ and $h_j(-n)$ is $l p\delta_{p,n} (h_i, h_j)$. Notice that every time we swap the position of $h_i(p)$ and $h_j(-n)$, the number of $+$ and $-$ in the commutator term are both lowered by 1. So at the end of the day when  all $h_i(p)$ with $p>0$ are positions before $\one$, the number of $+$ and $-$ in all the extra commutator terms are still kept the same. This shows that the elements are all in $V$. Thus we proved that $Y(u, x)v \in V[[x]]$. 
 
Now we argue the weak associativity. From Corollary 4.9 in \cite{H-MOSVA}, for every $a_1, ..., a_{k_1}, b_1, ..., b_{k_2} \in \{h_+, h_-\}$, $m_1, ..., m_{k_1}, n_1, ..., n_{k_2} \in \Z_+$, 
\begin{align*}
    & Y(a_1(-m_1)\cdots a_{k_1}(-m_{k_1})\one, x_1)Y(b_1(-n_1)\cdots b_{k_2}(-n_{k_2})\one, x_2)  \\
    & \quad = \sum_{i=0}^{\min\{k_1,k_2\}} \sum_{\substack{k_1\geq p_1 > \cdots > q_i \geq 1\\ 0\leq q_1 < \cdots < q_i \leq k_2}} l^i n_{q_1}\cdots n_{q_i}(a_{p_1}, b_{q_1})\cdots (a_{p_i}, b_{q_i})\\
    & \qquad \quad \cdot \binom{-n_{q_1}-1}{m_{p_1}-1}\cdots \binom{-n_{q_i}-1}{m_{p_i}-1} (x_1-x_2)^{-m_{p_1}-n_{q_1}-\cdots - m_{p_i}-n_{q_i}}\\
    & \qquad \quad \cdot \nord \left(\prod_{p\neq p_1, ..., p_i}\frac 1 {(m_p-1)!} \frac{\partial^{m_p-1}}{\partial x_1^{m_p-1}}a_p(x_1)\right)\left(\prod_{q\neq q_1, ..., q_i}\frac 1 {(n_q-1)!} \frac{\partial^{n_q-1}}{\partial x_2^{n_q-1}}b_q(x_2)\right)\nord,
\end{align*}
where the extra $l^i$ factor comes from a trivial generalization of Lemma 4.1 in \cite{H-MOSVA}. 

From Formula (5.29) of \cite{H-MOSVA}, we see that 
\begin{align*}
    & Y(Y(a_1(-m_1)\cdots a_{k_1}(-m_{k_1})\one, x_0)b_1(-n_1)\cdots b_{k_2}(-n_{k_2})\one, x_2) \\
    & \quad = \sum_{i=0}^{\min\{k_1,k_2\}} \sum_{\substack{k_1\geq p_1 > \cdots > q_i \geq 1\\ 0\leq q_1 < \cdots < q_i \leq k_2}} l^i n_{q_1}\cdots n_{q_i}(a_{p_1}, b_{q_1})\cdots (a_{p_i}, b_{q_i})\\
    & \qquad \quad \cdot \binom{-n_{q_1}-1}{m_{p_1}-1}\cdots \binom{-n_{q_i}-1}{m_{p_i}-1} x_0^{-m_{p_1}-n_{q_1}-\cdots - m_{p_i}-n_{q_i}}\\
    & \qquad \quad \cdot \nord \left(\prod_{p\neq p_1, ..., p_i}\frac 1 {(m_p-1)!} \frac{\partial^{m_p-1}}{\partial x_0^{m_p-1}}a_p(x_2+x_0)\right)\left(\prod_{q\neq q_1, ..., q_i}\frac 1 {(n_q-1)!} \frac{\partial^{n_q-1}}{\partial x_2^{n_q-1}}b_q(x_2)\right)\nord,\end{align*}
where negative powers of $x_2+x_0$ are expanded as a formal series in $x_2, x_0$ with lower truncated powers of $x_0$.

To show the weak associativity, we fix $i$ and $p_1, ..., p_i$ and compute the lower bound of power of $x_1$ of the series from the action of
$$\nord \left(\prod_{p\neq p_1, ..., p_i}\frac 1 {(m_p-1)!} \frac{\partial^{m_p-1}}{\partial x_0^{m_p-1}}a_p(x_1)\right)\left(\prod_{q\neq q_1, ..., q_i}\frac 1 {(n_q-1)!} \frac{\partial^{n_q-1}}{\partial x_2^{n_q-1}}b_q(x_2)\right)\nord$$
on an element $v = c_1(-r_1)\cdots c_{k_3}(-r_{k_3})\one$. 
For each $p \neq p_1, ..., p_i$, we compute the singular part of the term with respect to $a_p(x_1)$, namely,   $$\left(\frac{\partial^{m_p-1}}{\partial x_1^{m_p-1}}a_p(x_1)\right)^-= \sum_{s_p \geq 0}a_p (s_p)( - s_p - 1) \cdots ( - s_p - {m_p} + 1)x_1^{ - s_p - {m_p}}.$$
The term with the lowest power of $x_1$ will contain no $b_q(t)$ with $t>0$. Thus all such $a_p(s_p)$ would have the priority acting on $v$, resulting in 
$$\prod_{p\neq p_1, ..., p_i} a_p(s_p)c_1(-r_1)\cdots c_{k_3}(-r_{k_3})\one,$$
which is zero when
$$\sum_{p\neq p_1, ..., p_i} s_p > r_1 + \cdots + r_{k_3}.$$
So the power of $x_1$ is
\begin{align*}
    \sum_{p\neq p_1, ..., p_i} (-s_p-m_p) &\geq -r_1 -\cdots - r_{k_3} - \sum_{p\neq p_1, ..., p_i} m_p\\
    & \geq -r_1 -\cdots - r_{k_3} - m_1 - \cdots - m_{k_1}. 
\end{align*}
The lower bound we obtained at the right-hand-side works for all possible choices of $i$ and $p_1, ..., p_i$. Moreover, it depends only on the element $a_1(-m_1)\cdots a_{k_1}(-m_{k_1}) \one$ and $c_1(-r_1)\cdots c_{k_3}(-r_{k_2})\one$. So the pole-order condition is verified. 

Other axioms are verified similarly as in \cite{H-MOSVA}. 
\end{proof}

\begin{cor}
The graded dimension of the MOSVA $V(l, \one)$ is 
$$1 + \sum_{n=2}^\infty 2(n-1) \cdot {}_2F_1 \left(1-\frac n 2, \frac{3-n} 2; 2 ; 4\right)q^{n}.$$
\end{cor}

\begin{proof}
It suffices to argue the formula for $V(l, \one)$. Let $n\in \Z_+$. We compute $\dim V_{(2n)}$ and $\dim V_{(2n+1)}$ separately.

For each fixed $k\in \Z_+$, the set of ordered $k$-tuples $(m_1, ..., m_k)$ of positive numbers such that $m_1 + \cdots + m_k = 2n$ is well-known to be $\binom{2n-1}{k-1}$. Also within $i_1, ..., i_k$, the number of $+$ must be the same as the number of $-$. Thus $k$ must be an even number. Write $k = 2p$. Then the number of possible assignments of $+$ and $-$ to $i_1, ..., i_{2p}$ is simply $\binom{2p}p$. Summing up all possible choices of $p$, we have
$$\dim V_{(2n)} = \sum_{p=1}^n \binom{2p} p \binom{2n-1}{2p-1} = \sum_{p=1}^n \frac{2(2n-1)!}{p!(p-1)!(2n-2p)!} $$
Recall that in general, 
$${}_2F_1(a,b;c;z) = \sum \limits_{q = 0}^\infty  \frac{{a(a + 1) \cdots (a + q - 1)b(b + 1) \cdots (b + q - 1)}}{{c(c + 1) \cdots (c + q - 1)}}\frac{{{z^q}}}{{q!}}$$
Putting in $a=1-n, b=3/2-n, c=2, z=4$, we have
\begin{align*}
    & {}_2F_1(1-n, \frac 3 2 - n; 2; 4) \\
    & \quad = \sum_{q=0}^\infty \frac{(1-n)(1-n+1)\cdots (1-n+q-1)\cdot (\frac 3 2 - n)(\frac 3 2 - n + 1) \cdots (\frac 3 2 - n + q -1)}{2(2+1)\cdots (2+q-1)}\frac{2^q}{q!}\\
    & \quad = \sum_{q=0}^{n-1} \frac{(-1)^q(n-1)(n-2)\cdots (n-q)\cdot (3 - 2n)(3 - 2n + 2) \cdots (3 - 2n + 2q -2)}{(q+1)!}\frac{2^q}{q!}\\
    & \quad = \sum_{q=0}^{n-1} \frac{(2n-2)(2n-4)\cdots (2n-2q)\cdot(2n-3)(2n-5) \cdots (2n-2q-1)}{(q+1)!q!}\\
    & \quad = \sum_{q=0}^{n-1} \frac{(2n-2)!}{(2n-2q-2)!(q+1)!q!} = \sum_{p=1}^{n} \frac{(2n-2)!}{(2n-2p)!p!(p-1)!} 
\end{align*}
Thus 
$$\dim V_{(2n)}= 2(2n-1)\cdot {}_2F_1(1-n, \frac 3 2 - n; 2; 4)$$
Similarly, we compute that 
\begin{align*}
    \dim V_{(2n+1)} &= 2(2n)\cdot \sum_{p=1}^n \frac{(2n-1)!}{(2n-2p+1)!p!(p-1)!}\\
    &= 2(2n)\cdot {}_2 F_1(\frac 1 2 - n, 1-n; 2; 4)
\end{align*}
The conclusion then follows. 
\end{proof}

\begin{rema}
Note that in particular the weight-1 subspace is zero. 
If we understand $h_+(-m)$ and $h_-(-m)$ as the creation operators of certain physical objects (particles or strings), then the theorem simply says that these physical objects are always created in pairs. There does not exist one-object states in the MOSVA $V(l, \one)$.
\end{rema}

\begin{rema}
Note also that $V(l, \one)$ does not distinguish manifolds with different curvatures. Indeed, for any manifold with holonomy group $SO(2,\R)$, their MOSVAs are isomorphic.
\end{rema}


\section{Modules generated by eigenfunctions of the Laplace-Beltrami operator}

Let $V_U(l, f)$ be the $V(l, \one)$-module generated by a smooth function $f: U \to \C$ in \cite{H-MOSVA-Riemann} (cf. Section \ref{module-def}). In addition, we assume that the function $f$ satisfies
$$-\Delta f = \lambda f,$$
where $\Delta$ is the Laplace-Beltrami operator on $U$, $\lambda\in \C$. In this section, we study the structure of the module $V_U(l, f)$. In physics, eigenfunctions correspond to quantum states. So the module $V_U(l, f)$ can be understood as string-theoretic excitements of the quantum state corresponding to $f$. We will first deduce a lemma on covariant derivatives, then use it to show that every covariant derivative of an eigenfunction is a scalar multiple of the eigenfunction. Using this lemma, we identify a basis for $V_U(l, f)$. 

\subsection{Fundamental lemma of covariant derivatives}

In order to study the actions of $\Pi(T(E))$ on $f$, we will use the following theorem:

\begin{thm}\label{Comm-Cov}
Let $f: U\to \C$ be a $\C$-valued smooth function. Then for $n \geq 3$, we have
$$(\nabla^n f)(Z_1, ..., Z_{n-1}, Z_n ) - (\nabla^n f)(Z_1, ..., Z_n, Z_{n-1}) = 0.$$
And for $i = 1, ..., n-2$,
\begin{align*}
    & (\nabla^n f)(Z_1, ..., Z_i, Z_{i+1}, ..., Z_n) - (\nabla^n f)(Z_1, ..., Z_{i+1}, Z_{i}, ..., Z_n)\\
    = & \sum_{j=i+2}^n(\nabla^{n-2} f)(Z_1, ..., -R(Z_i, Z_{i+1})Z_{j}, ..., Z_n) \\
    = & (\nabla^{n-2} f)(Z_1, ..., -R(Z_i, Z_{i+1})Z_{i+2}, Z_{i+3} ..., Z_n) \\
    & + (\nabla^{n-2} f)(Z_1, ..., Z_{i+2}, -R(Z_i, Z_{i+1})Z_{i+3}, ..., Z_n) + \\
    & + \cdots \cdots \\
    & + (\nabla^{n-2} f)(Z_1, ..., Z_{i+2}, Z_{i+3}, ..., -R(Z_i, Z_{i+1})Z_n).
\end{align*}
\end{thm}

\begin{proof}
We prove the first equation by induction on $n$. For $n=3$, we have
\begin{align*}
    (\nabla^3 f)(Z_1, Z_2, Z_3) &= (\nabla_{Z_1} (\nabla^2 f))(Z_2, Z_3) \\
    & = \nabla_{Z_1} ((\nabla^2 f)(Z_2, Z_3)) - (\nabla^2 f)(\nabla_{Z_1}Z_2, Z_3) - (\nabla^2 f)(Z_2, \nabla_{Z_1}Z_3) \\
    & \quad \text{(note that $\nabla^2 f(X, Y) = \nabla^2 f(Y, X)$)} \\
    & = \nabla_{Z_1} ((\nabla^2 f)(Z_3, Z_2)) - (\nabla^2 f)(Z_3, \nabla_{Z_1}Z_2) - (\nabla^2 f)(\nabla_{Z_1}Z_3, Z_2)  = (\nabla^3 f)(Z_1, Z_3, Z_2).
\end{align*}
Assume the equation holds for $n-1$:
\begin{align*}
    (\nabla^n f)(Z_1, ..., Z_{n-1}, Z_n) & = (\nabla_{Z_1}(\nabla^{n-1}f))(Z_2, ..., Z_{n-1}, Z_n) \\ 
    & = \nabla_{Z_1}((\nabla^{n-1}f)(Z_2, ..., Z_{n-1}, Z_n)) - (\nabla^{n-1}f)(\nabla_{Z_1}Z_2, ..., Z_{n-1}, Z_n) \\
    & \quad - (\nabla^{n-1}f)(Z_2, ..., \nabla_{Z_1}Z_{n-1}, Z_n) - (\nabla^{n-1}f)(Z_2, ..., Z_{n-1}, \nabla_{Z_1}Z_n)\\
    & \quad \text{(by induction hypothesis)}\\
    & = \nabla_{Z_1}((\nabla^{n-1}f)(Z_2, ..., Z_n, Z_{n-1})) - (\nabla^{n-1}f)(\nabla_{Z_1}Z_2, ..., Z_n, Z_{n-1}) \\
    & \quad - (\nabla^{n-1}f)(Z_2, ..., Z_n, \nabla_{Z_1}Z_{n-1}) - (\nabla^{n-1}f)(Z_2, ..., \nabla_{Z_1}Z_n, Z_{n-1})\\
    & = (\nabla^n f)(Z_1, ..., Z_n, Z_{n-1}).
\end{align*}
So the first equation is proved. 

For the second equation, we first consider the case $i=1$:
\begin{align}
    &\quad (\nabla^n f)(Z_1, Z_2, Z_3, \cdots, Z_n)= (\nabla_{Z_1} (\nabla^{n-1}f))(Z_2, Z_3, \cdots, Z_n)\nonumber\\
    &= \nabla_{Z_1}((\nabla^{n-1}f) (Z_2, Z_3 ..., Z_n)) - (\nabla^{n-1}f)(\nabla_{Z_1}Z_2, Z_3, ..., Z_n) - \sum_{j=3}^n (\nabla^{n-1}f)(Z_2, ..., \nabla_{Z_1}Z_j, ..., Z_n) \nonumber\\
    &= \nabla_{Z_1}\nabla_{Z_2}((\nabla^{n-2}f) (Z_3, ..., Z_n)) - \sum_{j=3}^n \nabla_{Z_1}((\nabla^{n-2}f)(Z_3, ..., \nabla_{Z_2} Z_j, ..., Z_n))\label{Line1}\\
    & \quad -\nabla_{\nabla_{Z_1}Z_2} ((\nabla^{n-2}f)(Z_3, ..., Z_n))+\sum_{j=3}^n(\nabla^{n-2}f)(Z_3,..., \nabla_{\nabla_{Z_1}Z_2}Z_j, ..., Z_n)\label{Line2}\\
    & \quad 
    - \sum_{j=3}^n \left(\nabla_{Z_2}((\nabla^{n-2}f)(Z_3, ..., \nabla_{Z_1}Z_j, ..., Z_n) - \sum_{k=3}^{j-1}(\nabla^{n-2}f)(Z_3,..., \nabla_{Z_2}Z_k, ..., \nabla_{Z_1}Z_j, ..., Z_n)\right) \label{Line3} \\
    & \quad - \sum_{j=3}^n \left(-(\nabla^{n-2}f)(Z_3, ..., \nabla_{Z_2}\nabla_{Z_1}Z_j, ..., Z_n)-\sum_{k=j+1}^{n} (\nabla^{n-2}f)(Z_3, ..., \nabla_{Z_1} Z_j , ..., \nabla_{Z_2}Z_k, ..., Z_n)\right).\label{Line4}
\end{align}
Similarly, 
\begin{align}
    &\quad (\nabla^n f)(Z_2, Z_1, Z_3, \cdots, Z_n)= (\nabla_{Z_2} (\nabla^{n-1}f))(Z_1, Z_3, \cdots, Z_n)\nonumber\\
    & = \nabla_{Z_2}\nabla_{Z_1}((\nabla^{n-2}f) (Z_3, ..., Z_n)) - \sum_{j=3}^n \nabla_{Z_2}((\nabla^{n-2}f)(Z_3, ..., \nabla_{Z_1} Z_j, ..., Z_n))\label{Line5}\\
    & \quad -\nabla_{\nabla_{Z_2}Z_1} ((\nabla^{n-2}f)(Z_3, ..., Z_n))+\sum_{j=3}^n(\nabla^{n-2}f)(Z_3,..., \nabla_{\nabla_{Z_2}Z_1}Z_j, ..., Z_n)\label{Line6}\\
& \quad 
    - \sum_{j=3}^n \left(\nabla_{Z_1}((\nabla^{n-2}f)(Z_3, ..., \nabla_{Z_2}Z_j, ..., Z_n) - \sum_{k=3}^{j-1}(\nabla^{n-2}f)(Z_3,..., \nabla_{Z_1}Z_k, ..., \nabla_{Z_2}Z_j, ..., Z_n)\right)  \label{Line7}\\
    & \quad - \sum_{j=3}^n \left(-(\nabla^{n-2}f)(Z_3, ..., \nabla_{Z_1}\nabla_{Z_2}Z_j, ..., Z_n)-\sum_{k=j+1}^{n} (\nabla^{n-2}f)(Z_3, ..., \nabla_{Z_2} Z_j , ..., \nabla_{Z_1}Z_k, ..., Z_n)\right).\label{Line8}
\end{align}
Then in the difference, the second sum in (\ref{Line1}) cancels out with the first term in the sum of (\ref{Line7}); the first term in the sum of (\ref{Line3}) cancels out with the second sum in (\ref{Line5}); the second term in the sum of (\ref{Line3}), together with second term in the sum of (\ref{Line4}), cancel out those in (\ref{Line7}) and (\ref{Line8}). So the difference is
\begin{align*}
 & \quad (\nabla^n f)(Z_1, Z_2, Z_3, ..., Z_n) - (\nabla^n f)(Z_2, Z_1, Z_3, ..., Z_n)\\
    & = (\nabla_{Z_1}\nabla_{Z_2}-\nabla_{Z_2}\nabla_{Z_1})((\nabla^{n-2}f)(Z_3, ..., Z_n)) - \nabla_{\nabla_{Z_1}Z_2 - \nabla_{Z_2}Z_1} ((\nabla^{n-2}f)(Z_3, ..., Z_n) \\
    & \quad+ \sum_{j=3}^n (\nabla^{n-2}f)(Z_3, ..., \nabla_{\nabla_{Z_1}Z_2 -\nabla_{Z_2}Z_1} Z_j, ..., Z_n) + \sum_{j=3}^n (\nabla^{n-2}f)(Z_3, ..., (\nabla_{Z_2}\nabla_{Z_1}-\nabla_{Z_1}\nabla_{Z_2})Z_j, ..., Z_n)\\
    & = \sum_{j=3}^n (\nabla^{n-2}f)(Z_3, ..., (\nabla_{Z_2}\nabla_{Z_1}-\nabla_{Z_1}\nabla_{Z_2}+ \nabla_{\nabla_{Z_1}Z_2 - \nabla_{Z_2}Z_1})Z_j, ..., Z_n)\\
    &=\sum_{j=3}^n (\nabla^{n-2}f)(Z_3, ..., -R(Z_1, Z_2)Z_j, ..., Z_n).
\end{align*}
So the case $i=1$ is proved for arbitrary $n$. 

We proceed by induction of $i$. The base case has been proved above. Now we proceed with the inductive step. 
\begin{align*}
    & \quad (\nabla^n f)(Z_1, ..., Z_i, Z_{i+1}, ..., Z_n) = (\nabla_{Z_1}(\nabla^n f))(Z_2, ..., Z_i, Z_{i+1}, ..., Z_n) \\
    &= \nabla_{Z_1}((\nabla^{n-1} f)(Z_2, ..., Z_i, Z_{i+1}, ..., Z_n)) - \sum_{k=2}^{i-1}(\nabla^{n-1} f)(Z_2, ..., \nabla_{Z_1}Z_k, ..., Z_i, Z_{i+1}, ..., Z_n) \\
    & \quad - (\nabla^{n-1} f)(Z_2, ..., \nabla_{Z_1} Z_i, Z_{i+1}, ..., Z_n) - (\nabla^{n-1} f)(Z_2, ..., Z_i, \nabla_{Z_1} Z_{i+1}, ..., Z_n) \\
    & \quad - \sum_{k=i+2}^n  (\nabla^{n-1} f)(Z_2, ..., Z_i, Z_{i+1}, ..., \nabla_{Z_1}Z_k, ..., Z_n).
\end{align*}
Similarly, 
\begin{align*}
    & \quad (\nabla^n f)(Z_1, ..., Z_{i+1}, Z_i, ..., Z_n) = (\nabla_{Z_1}(\nabla^n f))(Z_2, ..., Z_{i+1}, Z_i, ..., Z_n) \\
    &= \nabla_{Z_1}((\nabla^{n-1} f)(Z_2, ..., Z_{i+1}, Z_i, ..., Z_n)) - \sum_{k=2}^{i-1}(\nabla^{n-1} f)(Z_2, ..., \nabla_{Z_1}Z_k, ..., Z_{i+1}, Z_i, ..., Z_n) \\
    & \quad - (\nabla^{n-1} f)(Z_2, ..., \nabla_{Z_1} Z_{i+1}, Z_i, ..., Z_n) - (\nabla^{n-1} f)(Z_2, ..., Z_{i+1}, \nabla_{Z_1} Z_i, ..., Z_n) \\
    & \quad - \sum_{k=i+2}^n  (\nabla^{n-1} f)(Z_2, ..., Z_{i+1}, Z_i, ..., \nabla_{Z_1}Z_k, ..., Z_n).
\end{align*}
We use the induction hypothesis to see that the difference is expressed as
\begin{align*}
    & \nabla_{Z_1}\left(\sum_{j=i+2}^n (\nabla^{n-3}f)(Z_2, ..., -R(Z_i, Z_{i+1})Z_j, ..., Z_n)  \right) \\ 
    & - \sum_{j=i+2}^n\sum_{k=2}^{i-1}(\nabla^{n-3}f)(Z_2, ..., \nabla_{Z_1}Z_k, ..., -R(Z_i, Z_{i+1})Z_j,..., Z_n) \\
    & - \sum_{j=i+2}^n(\nabla^{n-3}f)(Z_2, ..., -R(\nabla_{Z_1}Z_i, Z_{i+1})Z_j, ..., Z_n)\\
    & - \sum_{j=i+2}^n(\nabla^{n-3}f)(Z_2, ..., -R(Z_i, \nabla_{Z_1}Z_{i+1})Z_j, ..., Z_n)\\
    & -  \sum_{k=i+2}^n\sum_{j=i+2}^{k-1} (\nabla^{n-3}f)(Z_2, ..., -R(Z_i, Z_{i+1})Z_j, ..., \nabla_{Z_1}Z_k, ..., Z_n) \\
    & -\sum_{k=i+2}^n (\nabla^{n-3}f)(Z_2, ..., Z_{i+2},..., -R(Z_i, Z_{i+1})\nabla_{Z_1}Z_k, ..., Z_n)\\ 
    & -  \sum_{k=i+2}^n\sum_{j=k+1}^n(\nabla^{n-3}f)(Z_2, ..., Z_{i+2},..., \nabla_{Z_1}Z_k, ..., -R(Z_i, Z_{i+1})Z_j, ..., Z_n),
\end{align*}
which is equal to the right-hand-side. 
\end{proof}

\begin{rema}
We call Theorem \ref{Comm-Cov} the fundamental lemma of covariant derivatives, as it is of fundamental importance in this paper and has lots of important consequences. 
\end{rema}


\subsection{Zero-mode actions}
Now we use the lemma to compute $\Pi(T(E))f$. Recall the definition of the sectional curvature 
$$K = \frac{(R(U, V)V, U)}{(U, U)(V, V) - (U, V)^2},$$
where $U, V$ are any vector fields. Then by our assumption in Definition \ref{sections} $$(X_1, X_1) = (X_2, X_2) = 1, (X_1, X_2) = 0.$$
We compute directly that \begin{align*}
    (R(X_1, X_2)X_2, X_1) = K,\quad & (R(X_1, X_2)X_2, X_2) = 0,\\
    (R(X_1, X_2)X_1, X_1) = 0,\quad & (R(X_1, X_2)X_1, X_2) = -K.
\end{align*}
In other words, 
$$R(X_1, X_2)X_2 = KX_1, R(X_1, X_2)X_1 = -KX_2.$$
Therefore, 
\begin{align*}
    R(h_+,h_-)h_+ = 2K h_+, 
    R(h_+,h_-)h_- = -2K h_-.
\end{align*}

\begin{prop}\label{ConstMult}
Let $f$ be an eigenfunction for the Laplace-Beltrami operator of eigenvalue $\lambda$. Then for every $r\in \Z_+$, 
$$(\nabla^{2r} f)(h_+^{\otimes r} \otimes h_-^{\otimes r}) = \prod_{\alpha=1}^r (-\lambda + \alpha(\alpha-1)K) = (-\lambda)(-\lambda + 2K)\cdots (-\lambda + r(r-1)K),$$
$$(\nabla^{2r} f)(h_-^{\otimes r} \otimes h_+^{\otimes r}) = \prod_{\alpha=1}^r (-\lambda + \alpha(\alpha-1)K) = (-\lambda)(-\lambda + 2K)\cdots (-\lambda + r(r-1)K).$$
\end{prop}

\begin{proof}
We apply induction. If $r=1$, then \begin{align*}
    (\nabla^2f)(h_+\otimes h_-) &= (\nabla^2 f)((X_1-\sqrt{-1}X_2)\otimes(X_1 + \sqrt{-1} X_2)) \\
    & = (\nabla^2 f)(X_1 \otimes X_1 + X_2 \otimes X_2) + \sqrt{-1} (\nabla^2 f)(X_1 \otimes X_2 - X_2 \otimes X_1).
\end{align*} 
The second term is zero, while the first term is simply $\Delta f$. Thus we have 
$$(\nabla^2f)(h_+\otimes h_-) = -\lambda f.$$
So the base case is proved.

Assume the conclusion holds for all smaller $r$. We use Theorem \ref{Comm-Cov} to shift the $h_+$ and $h_-$ in the middle position
\begin{align*}
    \nabla^{2r}(h_+^{\otimes r} \otimes h_-^{\otimes r}) &= (\nabla^{2r}f)(h_+^{\otimes (r-1)}\otimes h_+ \otimes h_- \otimes h_-^{\otimes(r-1)}) \\
    & = (\nabla^{2r}f)(h_+^{\otimes (r-1)}\otimes h_-\otimes h_+ \otimes h_-^{\otimes (r-1)}) \\
    &\qquad - \sum_{p=0}^{r-2} (\nabla^{2r-2}f) (h_+^{\otimes(r-1)}\otimes h_-^{\otimes p} \otimes R(h_+,h_-)h_- \otimes h_-^{\otimes(r-2-p)})\\
    & = (\nabla^{2r}f)(h_+^{\otimes (r-1)}\otimes h_-\otimes h_+ \otimes h_-^{\otimes (r-1)}) +2K \sum_{p=0}^{r-2} (\nabla^{2r-2}f) (h_+^{\otimes(r-1)}\otimes h_-^{\otimes (r-1)})\\
    &= (\nabla^{2r}f)(h_+^{\otimes(r-1)}\otimes h_-\otimes h_+ \otimes h_-^{\otimes(r-1)}) + 2K(r-1) (\nabla^{2r-2})f(h_+^{\otimes(r-1)}\otimes h_-^{\otimes(r-1)}). 
\end{align*}
In other words, the price for passing $h_+$ from $r$-th position to $(r+1)$-th position is $$2K(r-1)\nabla^{2r-2}f(h_+^{\otimes (r-1)}\otimes h_-^{\otimes (r-1)}).$$ Similarly, the price of passing $h_+$ from $(r+1)$-th position to $(r+2)$-th position is 
$$2K(r-2)\nabla^{2r-2}f(h_+^{\otimes(r-1)}\otimes h_-^{\otimes(r-1)}).$$
We continue the process until $h_+$ arrives at $(2r-1)$-position and sum up the total price, to see that 
\begin{align}
    (\nabla^{2r}f)(h_+^{\otimes r}\otimes h_-^{\otimes r})& = (\nabla^{2r}f)(h_+^{\otimes (r-1)} \otimes h_-^{\otimes r-1} \otimes h_+ \otimes h_-) \nonumber \\
    & \qquad+ 2K( (r-1) + (r-2) + \cdots + 1)(\nabla^{2r-2}f)(h_+^{\otimes (r-1)}\otimes h_-^{\otimes (r-1)}).\label{Lem4-2}
\end{align}
Recall Theorem 4.1 in \cite{H-MOSVA-Riemann}: if $X$ and $Y$ are two parallel tensors of degree $m$ and $n$, then 
$$(\nabla^{m+n} f)(X, Y) = (\nabla^n (\nabla^m f(Y)))(X).$$
Thus with the conclusion of the base case, the first term on the right-hand-side of (\ref{Lem4-2}) is simply $$\nabla^{2r-2}[(\nabla^2 f)(h_+\otimes h_-)](h_+^{\otimes(r-1)}\otimes h_-^{\otimes (r-1)}) = -\lambda(\nabla^{2r-2}f)(h_+^{\otimes(r-1)}\otimes h_-^{\otimes (r-1)}). $$
Computing the second term and combine it back to (\ref{Lem4-2}), we see that
\begin{align}
    (\nabla^{2r}f)(h_+^{\otimes r}\otimes h_-^{\otimes r})& = (-\lambda+ K r(r-1)) (\nabla^{2r-2}f)(h_+^{\otimes (r-1)}\otimes h_-^{\otimes (r-1)}).\label{Lem4-2-1}
\end{align}
The first conclusion then follows from induction hypothesis. 

The second conclusion follows from an almost identical argument. We shall not repeat the details here. 
\end{proof}

\begin{prop}\label{4-4} 
Let $f$ be an eigenfunction for the Laplace-Beltrami operator of eigenvalue $\lambda$. Fix any $i_1, ..., i_{2r} \in \{+, -\}$ such that $\#\{j: i_j = +\} = \#\{j: i_j = -\} = r$. 
\begin{enumerate}
    \item There exists a single-variable polynomial $P$, such that 
    $$(\nabla^{2r}f)(h_{i_1}\otimes \cdots \otimes h_{i_{2r}}) = P(\lambda)f.$$
    \item The roots of $P(\lambda)$ are contained in \begin{align}
    \{p(p-1)K, p = 1,..., r\}.\label{setofroot}
    \end{align}
    \item If $i_{2r-t+1}=\cdots = i_{2r}$, then for every $p\leq t$, $p(p-1)K$ are roots of $P(\lambda)$ . 
\end{enumerate}
\end{prop}

\begin{proof}
We argue by induction on $r$. For $r=1$, there are only two cases of $(i_1,i_2)$: $(+, -)$ and $(-, +)$. $P(\lambda)=-\lambda$ works for both cases. (2) and (3) are obvious. 

Assume that all conclusions hold for strictly smaller $r$. Without loss of generality, we assume that $i_{2r} = i_{2r-1} = \cdots = i_{2r-t+1}= -$, $i_{2r-t} = +$ for some $t \in [1,r]$. In other words, there are $t$ consecutive $h_-$ at the rear, precede by a $h_+$. Theorem \ref{Comm-Cov} allows us to move the $h_+$ from the $(2r-t)$-th position to the $(2r-1)$-position. By a computation similar to that for (\ref{Lem4-2-1}), we have 
\begin{align*}
    & (\nabla^{2r}f)(h_{i_1}\otimes \cdots \otimes h_{i_{2r-t-1}} \otimes h_+ \otimes h_-\otimes h_- \otimes \cdots \otimes h_-)\\
    & = (-\lambda + Kt(t-1))(\nabla^{2r-2}f)((h_{i_1}\otimes \cdots \otimes h_{i_{2r-t-1}} \otimes \widehat{h_+} \otimes \widehat{h_-} \otimes h_-\otimes \cdots \otimes h_-) .
\end{align*}
Here the hat notation is introduced to show the removed terms. By the induction hypothesis, there exists a polynomial $Q(\lambda)$ satisfying (1), (2) and (3). So
\begin{align*}
    (\nabla^{2r}f)(h_{i_1}\otimes \cdots \otimes h_{i_{2r-t-1}} \otimes h_+ \otimes h_-\otimes h_- \otimes \cdots \otimes h_-) = (-\lambda + Kt(t-1))Q(\lambda)f.
\end{align*}
Therefore (1) holds with $P(\lambda) = (-\lambda + Kt(t-1))Q(\lambda)$. The roots of $P(\lambda)$, by induction hypothesis, are contained in $\{p(p-1)K: p= 1, ..., r-1\}\cup \{t(t-1)K\}$. Since $t\in [1, r]$, (2) holds for $P(\lambda)$. For $j_1 = i_1, ..., j_{2r-t}= i_{2r-t}, j_{2r-t+1} = \cdots = j_{2r-2} = -$, the induction hypothesis shows that for every $\alpha = 1, ..., t-1$, $\alpha(\alpha-1)K$ are roots for $Q(\lambda)$. Therefore (3) holds, with the additional root $t(t-1)K$ for $P(\lambda)$. 
\end{proof}

Since $\Pi(T(E))$ is spanned by $h_{i_1} \otimes \cdots \otimes h_{i_{2r}}$ in Proposition \ref{4-4}, we have thus proved the following theorem: 

\begin{thm}\label{ParallelScalar}
Let $f$ be an eigenfunction of the Laplace-Beltrami operator. Then as a vector space, $$\Pi(T(E))f = \C f.$$
In other words, the action of every parallel tensor on an eigenfunction $f$ generates only scalar multiples of $f$. 
\end{thm}

\begin{rema}
In case $f$ is an eigenfunction with real eigenvalue, then the same argument above shows that $\Pi(T(TM))f \in \R f$. 
\end{rema}

\begin{rema}
While eigenfunctions defined globally on a manifold are known to have real and nonpositive eigenvalues, we would like to note that eigenfunctions with imaginary eigenvalues do exist locally. Indeed, let $M$ be the two-dimensional unit sphere with coordinates as in Example \ref{3-8}. Then the Laplace-Beltrami operator is represented by 
$$\Delta = \frac{\partial^2}{\partial\phi^2} + \cot \phi \frac{\partial}{\partial \phi} + \frac{1}{\sin^2\phi}\frac{\partial^2}{\partial\theta^2}. $$
Consider a function $f(\phi, \theta) = u(\phi) + \sqrt{-1} v(\phi)$, where $u,v$, together with their derivative $u', v'$, satisfies following linear system of ODE
$$\frac{d}{d\phi}\begin{bmatrix}
u \\ u' \\ v \\ v'
\end{bmatrix} = 
\begin{bmatrix}
0 & 1 & 0 & 0 \\
a & -\cot\phi & -b & 0\\
0 & 0 & 0 & 1 \\
b & -\cot\phi & a & 0
\end{bmatrix}\begin{bmatrix}
u \\ u' \\ v \\ v'
\end{bmatrix},$$
with some initial conditions specified $\phi=\pi/4$. Since all the coefficients are smooth near $\pi/4$, the solution exists smoothly in $(\pi/4-\epsilon, \pi/4+\epsilon)$ for some $\epsilon > 0$. Thus $f$ is a smooth complex-valued function defined in the open subset $\{(\phi, \theta): \pi/4-\epsilon < \phi < \pi/4 + \epsilon, 0 \leq \theta < 2\pi\}$. It is routine to check that $\Delta f = (a+\sqrt{-1}b)f$. 
\end{rema}

\begin{rema}
Theorem \ref{ParallelScalar} can be generalized to higher dimensional orientable and non-orientable space forms. See \cite{Q-Cov-Der}. 
\end{rema}

\begin{rema}\label{genspe}
We call an eigenvalue $\lambda$ special if $\lambda = \alpha(\alpha-1)$ for some $\alpha \in \Z_+$, generic if otherwise. Proposition \ref{4-4} in particular shows if $f$ is a global eigenfunction over the space form with negative section curvature, then the action of individual parallel tensors $h_{i_1}\otimes \cdots \otimes h_{i_{2r}}$ are nonvanishing. On the other hand, it is well known that the eigenvalues of global eigenfunctions over the two-dimensional unit sphere are all special. The difference of generic eigenvalues and special eigenvalues will also be reflected on the irreducible $V$-modules generated by eigenfunctions, as will be seen later in the paper. 
\end{rema}

\subsection{Structure of the module}

Having determined the action of $\Pi(T(E))$ on $f$, we now explicitly identify the module for the MOSVA constructed by Huang in \cite{H-MOSVA-Riemann} (cf. Section \ref{module-def}. The following theorem by Dong, Li, and Mason will be needed in the discussion:

\begin{lemma}[\cite{LL}, Proposition 4.5.6, \cite{DLM}] Let $W$ be a $V$-module and let $T$ be a subset of $W$. Then the submodule generated by $T$ is spanned by 
$$\{v_nw: v\in V, n\in \Z, w\in T\}.$$
\end{lemma}
Though the lemma was formulated for vertex algebras, the proof uses only weak associativity and thus applies to modules for MOSVAs (see \cite{LL}, Proposition 4.5.7). 


Now we state a general theorem regarding the spanning set of a module generated by a lowest weight element. 

\begin{thm}\label{4-10}
Let $W$ be a module for $V(l, \one)$ generated by an element $w$ of lowest weight $\mu$, i.e., for every $n\in \Z_+$, $W_{[\mu - n]} =0$. Then $W$ has the following spanning set
\begin{align}\left\{ \sum_{i} X_{(1)i}(-t_1) \cdots X_{(k)i}(-t_k) X_{(k+1)i}(0) \cdots X_{(r)i}(0)w: \begin{aligned}&r \geq 0, k = 1, ..., r, t_1, .., t_k > 0;\\
& \sum_{i} X_{(1)i} \otimes \cdots \otimes X_{(r)i} \in \Pi(E^{\otimes r})\end{aligned}\right\}.\label{spanset}
\end{align}
Here the $\otimes$ symbol is omitted for convenience. 
\end{thm}

\begin{proof}%
Consider now the action of $\sum_{i}X_{(1)i}(-m_1)\cdots X_{(r)i}(-m_r) \one$. We see that
\begin{align*}
    & \quad Y\left(\sum_{i}X_{(1)i}(-m_1)\cdots X_{(r)i}(-m_r) \one, x\right)w \\
    & = \sum_{n_1, ..., n_r \in \Z} \left(\prod_{j=1}^r \frac{(-n_j-1) \cdots (-n_j-m_j+1)}{(m_j-1)!}\right)\left(\nord \sum_{i} X_{(1)i}(n_1) \cdots X_{(r)i}(n_r)\nord w\right) x^{\sum_{j=1}^r(-n_j-m_j)}.
\end{align*}
Since $w$ is of lowest weight in $W$, 
if one of the $n_1, ..., n_r$ is positive, then after taking the normal ordering, the coefficient is zero. So it suffices to focus on nonpostive $n_1, ..., n_r$. Then $W$ is spanned by the following elements
\begin{align}
    \sum_{\substack{n_1 + \cdots + n_r = -p, \\ n_1, ..., n_r \leq 0} } \left(\prod_{j=1}^r \frac{(-n_j-1) \cdots (-n_j-m_j+1)}{(m_j-1)!}\right) \left(\nord\sum_{i} X_{(1)i}(n_1)\cdots X_{(r)i}(n_r)\nord w\right)\label{genericelement}
\end{align}
for $r\geq 0, p\geq 0, m_1, ..., m_r > 0$, and $\sum_{i} X_{(1)i} \otimes \cdots \otimes X_{(r)i} \in \Pi(E^{\otimes r})$. 

For each fixed $n_1, ..., n_r\leq 0$, 
$$\nord \sum_i X_{(1)i}(n_1)\cdots X_{(r)i}(n_r) \nord f=\sum_i  X_{(\sigma(1))i}(n_{\sigma(1)}) \cdots X_{(\sigma(k))i)}(n_{\sigma(k)}) X_{(\sigma(k+1))i}(0)\cdots X_{(\sigma(r))i}(0)f,$$
where $\sigma$ is the unique permutation such that \begin{align*}
    \sigma(1)< \cdots < \sigma(k)&,\quad \sigma(k+1)<\cdots <\sigma(r);\\
    n_{\sigma(1)}, ..., n_{\sigma(k)}<0&,\quad n_{\sigma(k+1)} = \cdots = n_{\sigma(r)} = 0.
\end{align*} 
Note that $\sum_i X_{(\sigma(1))i}\otimes \cdots \otimes X_{(\sigma(r))i}$ stays in $\Pi(E^{\otimes r})$, thus the summand of (\ref{genericelement}) for each fixed $n_1,..., n_r$ is a vector in (\ref{spanset}). Therefore, (\ref{genericelement}) is a sum of vectors in (\ref{spanset}). So $W$ is a subset of the linear span of (\ref{spanset}). 

Now we argue that every vector in (\ref{spanset}) is in $W$. First we notice that for each $j=1,...,r$
$$\frac{(-n_j-1) \cdots (-n_j-m_j+1)}{(m_j-1)!}\neq 0 \Rightarrow n_j = 0 \text{ or } n_j \leq -m_j.$$
This allows us to represent each vector in (\ref{spanset}) as a linear combination of elements of the form (\ref{genericelement}) by choosing $m_1, ..., m_r$ and $p$ appropriately. We show this by induction on $k$. In case $k = 1$, we pick $m_1 = t_1, m_2 = \cdots = m_r = t_1+1$ and $p=t_1$. Then every nonzero summand in (\ref{genericelement}) is given by $(n_1, ..., n_r)$ such that 
$$\begin{aligned}
& n_1+\cdots + n_r = -t_1, \\
& n_1 = 0 \text{ or } n_1 \leq -t_1, \\
& n_i = 0 \text{ or } n_i \leq -(t_1 + 1), i = 2, ..., r.
\end{aligned}$$ 
The only solution is $(n_1, ..., n_r) = (-t_1, 0, ..., 0)$. Thus $\sum_i X_{(1)i}(-t_1)X_{(2)i}(0)\cdots X_{(r)i}(0)f$ is an element of $W$. The base case is proved. 

Now assume every element in (\ref{spanset}) of smaller $k$ is represented by a linear combination of elements of the form (\ref{genericelement}). We pick $m_1 = t_1, ..., m_k = t_k, m_{k+1} =\cdots = m_r = t_1 + \cdots + t_k + 1$ and $p=t_1+\cdots + t_k$. Then every nonzero summand in (\ref{genericelement}) is given by $(n_1, ..., n_r)$ such that 
\begin{align*}
    & n_1 + \cdots + n_r = -(t_1 + \cdots + t_k), \\
    & n_i = 0 \text{ or }n_i \leq -t_i, i = 1, ..., k, \\
    & n_i = 0 \text{ or } n_i \leq -(t_1 +\cdots + t_k + 1), i = k+1, ..., r. 
\end{align*}
Necessarily, $n_{k+1}= \cdots = n_{r}=0$. For $n_1, ..., n_k$, aside from the choice $n_1=-t_1, ..., n_k = -t_k$, all other choices would involve some zeros. By induction hypothesis, the summands in (\ref{genericelement}) given by these extra choices are all contained in $W$. Subtracting (\ref{genericelement}) by these summands, we obtain $\sum_i X_{(1)i}(-t_1)\cdots X_{(k)i}(-t_k)X_{(k+1)i}(0) \cdots X_{(r)i}(0)$, which is then in $W$. 
\end{proof}

\begin{cor}\label{4-11}
Let $f: U\to \C$ be an eigenfunction for the Laplace-Beltrami operator over an open subset $U$ of $M$. Let $V_U(l, f)$ be the $V(l, \one)$-module constructed by Huang in \cite{H-MOSVA-Riemann} (cf. Section \ref{module-def}) with $\wt f = 0$. Then the following set
\begin{align}\left\{  h_{i_1}(-t_1) \cdots h_{i_k}(-t_k) h_{i_{k+1}}(0) \cdots h_{i_{r}}(0)f: \begin{aligned} & r\geq 0, k = 1, ..., r, t_1, .., t_k > 0; \\
&\#\{j\in [1,r]: i_j = +\} = \#\{j\in [1,r]: i_j = -\}; \\
&\forall p = k+1, ..., r, \\
& \#\{j\in [p,r]: i_j = + \}\neq \#\{j\in [p,r]: i_j = - \}\end{aligned}\right\}.\label{spanset-dim2}
\end{align}
forms a basis for $V_U(l, f)$. 
\end{cor}

\begin{proof}
With the same argument above, we see that elements in (\ref{spanset-dim2}) without the third and fourth requirements forms a spanning set of $V_U(l, f)$. The third and fourth requirements are introduced to exclude the relation brought by $\Pi(T(E))f = \C f$. It is clear that vectors in (\ref{spanset-dim2}) are linearly independent in $T(\widehat{E_p}_-)\otimes T(E_p) \otimes_{\Pi(T(E))} C^\infty(U)$. 
\end{proof}

\begin{rema}
Obviously, $V_U(l, f)$ is not grading-restricted, as one can insert arbitrarily many parallel sections between the $k$-th and $(k+1)$-th positions. The resulted set of new tensors are all linearly independent to each other in $T(\widehat{E}_-)_p\otimes T(E)_p \otimes_{\Pi(T(E))} C^\infty(U)$. Nevertheless, the irreducible quotients of $V_U(l, f)$ are grading-restricted, as shown later. 
\end{rema}

\subsection{Isomorphic relations} We know that two orientable space forms with different sectional curvatures cannot be distinguished by the MOSVAs. Now we show that they can be distinguished by the modules. 

\begin{thm}\label{isom-reln}
Let $M_1$, $M_2$ be orientable space forms of the same dimension. Let $K_1$ and $K_2$ be their sectional curvatures. Let $U_1$ and $U_2$ be open subsets of $M_1$ and $M_2$, $f_1: U_1 \to \C, f_2: U_2 \to \C$ be eigenfunctions for the Beltrami-Laplace operator of eigenvalues $\lambda_1$ and $\lambda_2$. Then 
\begin{enumerate}
    \item $\C f_1$ and $\C f_2$ are isomorphic as $\Pi(T(E))$-modules if and only if $\lambda_1 = \lambda_2, K_1 = K_2$. 
    \item $V_{U_1}(l, f_1)$ and $V_{U_2}(l, f_2)$ are isomorphic as $V(l, \one)$-modules if and only if $\lambda_1 = \lambda_2, K_1 = K_2$. 
\end{enumerate}
\end{thm}

\begin{proof}
For Part (1), without loss of generality, let $T f_1 = f_2$. We argue that $T$ is an $\Pi(T(E))$-isomorphism if and only if $\lambda_1 = \lambda_2, K_1 = K_2$. 

The only if part can be seen by the action of $h_+\otimes h_-$ and $h_+\otimes h_+\otimes h_-\otimes h_-$. Recall that for any eigenfunction $f$ of eigenvalue $\lambda$ over an open subset of an orientable space form with sectional curvature $K$ 
\begin{align*}
    &(h_+\otimes h_-)f = (\nabla^2f)(h_+\otimes h_-) = -\lambda f\\
    &(h_+\otimes h_+\otimes h_-\otimes h_-)f = (\nabla^4 f)(h_+\otimes h_+\otimes h_-\otimes h_-) = -\lambda(-\lambda+2K)f
\end{align*}
Thus
\begin{align*}
    &T((h_+\otimes h_-) f_1)  = (h_+\otimes h_-) T(f_1) \Rightarrow \lambda_1 f_2 = \lambda_2 f_2 \Rightarrow \lambda_1 = \lambda_2\\
    &T((h_+\otimes h_+\otimes h_-\otimes h_-) f_1)   =  (h_+\otimes h_+\otimes h_-\otimes h_-)T(f_1) \\
    &\Rightarrow \lambda_1^2- \lambda_1 K_1 (\dim M-1) = \lambda_2^2 - \lambda_2 K_2(\dim M - 1)
     \Rightarrow  K_1 = K_2.
\end{align*}
The if part follows from an  argument similar to the process of Proposition \ref{ConstMult} and Proposition \ref{4-4}, showing that 
$$T((h_{i_1}\otimes \cdots \otimes h_{i_r})f_1)= (h_{i_1}\otimes \cdots \otimes h_{i_r}) T(f_1)$$
for every $i_1, ..., i_r\in \{+, -\}$ with $\#\{j: i_j = +\} = \#\{j: i_j = -\}$. 
We shall not repeat here. 

For Part (2), if $\lambda_1 = \lambda_2$ and $K_1 = K_2$, then $\C f_1$ is isomorphic to $\C f_2$ as $\Pi(T(E))$-modules. We then see that the induced module $T(E)_p\otimes_{\Pi(T(E))} f_1$ is isomorphic to the induced module $T(E_p)\otimes_{\Pi(T(E))} f_2$ as $N_p(E)$-modules. 
From the construction of the modules $V_U(l, f_1)$ and $V_U(l, f_2)$ in Section 2.4 and \cite{H-MOSVA-Riemann}, they are isomorphic as $V(l, \one)$-modules. 

Conversely, if $V_{U_1}(l, f_1)$ is isomorphic to $V_{U_2}(l, f_2)$, we can similarly consider the actions of $h_+(-1)h_+(-1)\one$ and $h_+(-1)h_+(-1)h_-(-1)h_-(-1)\one$. The zero-modes of these action are precisely $h_+(0)\otimes h_-(0)$ and $h_+(0)\otimes h_+(0)\otimes h_-(0)\otimes h_-(0)$ discussed in Part (1). The same discussion gives $\lambda_1 = \lambda_2$ and $K_1 = K_2$. 
\end{proof}

\section{Irreducible modules generated by eigenfunctions} 

In this section, we study quotients of $V_U(l,f)$, where $U$ is an open subset of an orientable space form, $f: U\to \C$ is an eigenfunction for the Beltrami-Laplace operator of eigenvalue $\lambda$. 

Since $V_U(l, f)$ is generated by the unique (up to a scalar multiple) lowest weight element $f$, a nonzero submodule is proper if and only if it does not contain $f$. Since the homogeneous subspace of weight $\lambda$ is of one dimension, the sum of two proper submodules does not contain $f$ and thus stays as a proper submodule. Thus we conclude the following proposition: 

\begin{prop}
There exists a unique maximal proper submodule of $V_U(l, f)$. Thus $V_U(l, f)$ has a unique irreducible quotient, called the irreducible module generated by $f$. 
\end{prop}

\subsection{Lowest weight projection formula} The main tool to locate the irreducible quotient is an explicit formula of the projection $Y(v,x)w$ to the lowest weight space of $V_U(l, f)$ for every homogeneous $v\in V(l, 1)$ and $w\in V_U(l, f)$. In other words, the formula gives the coefficient of $x^{-n-1}$ in $Y(v, x)$ where $v_nw$ is of the same weight of $f$. This coefficient will be called the lowest weight projection of $Y(v, x)w$. 

\begin{nota}
To avoid using quadruple subscripts, in this section, we use $\state{+}$ to denote the vector field $h_+$, $\state{-}$ to denote the vector field $h_-$. Let $i_1,..., i_s\in \{+, -\}$. We denote each $h_{i_j}$ by $\state{i_j}$. For $m_1, ..., m_s \in \Z$, the operator $h_{i_1}(m_1)\cdots h_{i_s}(m_s)$ will be denoted by $\state{i_j(m_1)\cdots i_s(m_s)}$, or $\state{i_1(m_1)\cdots i_k(m_k)}\cdot\state{i_{k+1}(m_{k+1}) \cdots i_{s}(m_s)}$, or $\state{i_1(m_1)\cdots i_k(m_k)}\state{i_{k+1}(m_{k+1}) \cdots i_{s}(m_s)}$, for every $k=1, ..., s$. 
In case $m_1= \cdots =m_s = 0$, we may also use $\state{i_1\cdots i_r}$ to denote the operator $h_{i_1}(0)\cdots h_{i_r}(0)$, omitting the zero indicator. But for some situations we will still keep $\state{i_1(0)\cdots i_r(0)}$ without the abbreviation. The inner product $(h_{i_j}, h_{k_s})$ will simply be denoted as $\langle i_j, k_s\rangle$. It follows from a direct computation that 
\begin{align*}
\langle +, + \rangle & = (h_+, h_+) = 0, \langle -, - \rangle = (h_-, h_-) = 0\\
\langle +, - \rangle & = (h_+, h_-) = 2, \langle -, + \rangle = (h_-, h_+) = 2. 
\end{align*}
\end{nota}

\begin{prop}\label{proj0-prop}
Let $m_1, ...,, m_s, t_1, ..., t_k \in \Z_+$. Then for 
\begin{align*}
    v &=\state{i_1(-m_1)\cdots i_s(-m_s)}\one,\\
    w &= \state{j_1(-t_1)\cdots j_k(-t_k)j_{k+1}(0)\cdots j_r(0)}f.
\end{align*}
Then for
$$n = m_1 + \cdots + m_s + t_1 + \cdots + t_k - 1,$$
$v_n w$ is the lowest weight component in $Y(v, x)w$. If $s<k$, then $v_n w = 0$; if $s\geq k$, then 
$v_n w$ is equal to
\begin{align}
\sum_{1\leq c_k < \cdots < c_1 \leq s} \sum_{\sigma\in S_k}
\prod_{\substack{1\leq p \leq s\\
p\neq c_1, ..., c_k}}(-1)^{m_p - 1}&\cdot \prod_{p=1}^k \frac{(-t_{\sigma(p)}-1) \cdots (-t_{\sigma(p)}-m_{c_p}+1)}{(m_{c_p}-1)!} \cdot \prod_{p=1}^k t_{\sigma(p)}l \langle i_{c_p}, j_{\sigma(p)} \rangle\nonumber \\
& \cdot \state{i_1(0) \cdots \widehat{i_{c_k}(0)} \cdots \widehat{i_{c_1}(0)} \cdots i_s(0) j_{k+1}(0)\cdots j_r(0)}f. \label{proj0-formula}
\end{align}
\end{prop}

\begin{proof}
Since $\wt v = m_1 + \cdots m_s, \wt w = t_1 + \cdots + t_k$, we know that $\wt v_n w = \wt f$ if and only if 
$$m_1 + \cdots + m_s + t_1 + \cdots + t_k - n - 1 = 0. $$
Thus 
$$n = m_1 + \cdots + m_s + t_1 + \cdots + t_k -1. $$
By definition, 
$$Y(v, x)w = \sum_{n_1, \cdots n_s \in \Z} \prod_{p=1}^s \frac{(-n_p - 1)\cdots (-n_p-m_p + 1)}{(m_p-1)!} \nord \state{i_1(n_1) \cdots i_s(n_s)}\nord w x^{-n_1-m_1 - \cdots - n_s-m_s}.$$
Thus for $n$ specified above, the coefficient $x^{-n-1}$ in the series can be simplified as 
\begin{align}
    \sum_{\substack{n_1+\cdots+n_s = t_1+\cdots + t_k\\
n_1, ..., n_s\in \Z_+\cup \{0\}}} \prod_{p=1}^s \frac{(-n_p - 1)\cdots (-n_p-m_p + 1)}{(m_p-1)!} \nord \state{i_1(n_1) \cdots i_s(n_s)}\nord w. \label{proj0-origin}
\end{align}
Here $n_1, ..., n_s\in \Z_+ \cup \{0\}$ because any occurrence of negative number will result in a zero summand. The normal ordering originally pushes all the zero-modes to the right. However, since the zero-modes commute with positive-modes, we can instead have all the positive-modes first act on $w$, then compose with the zero-modes. 

Now for any $n>0$ and any $i\in \{+,-\}$, we study the action of $i(n)$ on $w$. Using the fact that 
$$\state{i(n)j(-t)} = \state{j(-t)i(n)} + t\delta_{nt}l\langle i,j\rangle,$$
we see that
\begin{align*}
    \state{i(n) w} &= \state{i(n) j_1(-t_1)\cdots j_k(-t_k)j_{k+1}(0)\cdots j_r(0)}f\\
    &= \sum_{q=1}^k t_q \delta_{n,t_q} \langle i,j_q\rangle \state{j_1(-t_1)\cdots \widehat{j_q(-t_q)}\cdots j_k(-t_k)j_{k+1}(0)\cdots j_r(0)}f.
\end{align*}
In particular, the action is nonzero only when $n$ coincides with one of the $t_q$'s. For all other choices of $n$, we simply get zero. 

Based on the above observations, we see that (\ref{proj0-origin}) is nonzero only when $s\geq k$. To evaluate (\ref{proj0-origin}), we first choose $k$ elements $c_1> ...> c_k$ in $\{1, ..., s\}$ and specify $i_{c_1}(n_{c_1}), ..., i_{c_k}(n_{c_k})$ as positive-modes. Then (\ref{proj0-origin}) becomes
\begin{align}
    \sum_{1\leq c_k <\cdots < c_1 \leq s}\sum_{\substack{n_{c_1}+\cdots + n_{c_k} = t_1+\cdots + t_k\\
    n_1, ..., n_{c_k} \in \Z_+}} & \prod_{p=1}^k \frac{(-n_{c_p}-1)\cdots (-n_{c_p}-m_{c_p}-1)}{(m_{c_p}-1)!} \prod_{\substack{1\leq p \leq s \\ p \neq c_1, ..., c_k}}(-1)^{m_p-1} \nonumber \\
    & \cdot \state{i_{1}(0) \cdots \widehat{i_{c_k}(0)} \cdots \widehat{i_{c_1}(0)} \cdots i_s(0)}
    \cdot \state{i_{c_k}(n_{c_k})\cdots i_{c_1}(n_{c_1})}w. \label{proj0-1}
\end{align}
In order that $\state{i_{c_k}(n_{c_k})\cdots i_{c_1}(n_{c_1})}w \neq 0$, $n_{c_1}, ..., n_{c_k}$ has to be chosen among $\{t_1, ..., t_k\}$. One computes that $$\state{i_{c_k}(n_{c_k})\cdots i_{c_1}(n_{c_1})}w = \sum_{\sigma\in S_k}t_{\sigma(1)}l\delta_{n_{c_1},t_{\sigma(1)}}\langle i_{c_1}, j_{\sigma(1)}\rangle \cdots t_{\sigma(k)}l\delta_{n_{c_k}, t_{\sigma(k)}}\langle i_{c_k}, j_{\sigma(k)}\rangle \state{j_{k+1}(0) \cdots j_{r}(0)}f.$$
Thus (\ref{proj0-1}) becomes
\begin{align*}
    \sum_{1\leq c_k < \cdots < c_1 \leq s}\sum_{\sigma\in S_k} & \prod_{p=1}^k \frac{(-t_{\sigma(p)}-1)\cdots (-t_{\sigma(p)}-m_{c_p}+1)}{(m_{c_p}-1)!}\prod_{\substack{1\leq p \leq s\\ p\neq c_1, ..., c_k}}(-1)^{m_p-1}\\ 
    & \cdot \prod_{p=1}^k t_{\sigma(p)} l \langle i_{c_p}, j_{\sigma(p)}\rangle  \state{i_1(0)\cdots \widehat{i_{c_k}(0)} \cdots \widehat{i_{c_1}(0)} \cdots i_s(0) j_{k+1}(0)\cdots j_r(0)}f.
\end{align*}
The conclusion then follows. 
\end{proof}

\begin{rema}
Since for every $t>0$, 
$$\frac{(-t-1)\cdots (-t-m+1)}{(m-1)!} = (-1)^{m-1}\binom{t+m-1}{t},$$
formula (\ref{proj0-formula}) can also be written as 
\begin{align}
\sum_{1\leq c_k < \cdots < c_1 \leq s} \sum_{\sigma\in S_k}
\prod_{\substack{1\leq p \leq s}}(-1)^{m_p - 1}&\cdot \prod_{p=1}^k \binom{t_{\sigma(p)} + m_{c_p} - 1}{t_{\sigma(p)}} \cdot \prod_{p=1}^k t_{\sigma(p)}l \langle i_{c_p}, j_{\sigma(p)} \rangle\nonumber \\
& \cdot \state{i_1(0) \cdots \widehat{i_{c_k}(0)} \cdots \widehat{i_{c_1}(0)} \cdots i_s(0) j_{k+1}(0)\cdots j_r(0)}f. \label{proj0-formula-2}
\end{align}
\end{rema}


\subsection{First quotient}

We first locate a submodule large enough that the quotient by which is grading-restricted. 

\begin{prop}\label{5-5}
For every $j_1, ..., j_r$, there exists a constant
$C=C_{j_1,...,j_r}(\lambda)$ depending only on $j_{1},..., j_r$ and the eigenvalue $\lambda$, such that for every $i_1, ..., i_s\in \{+, -\}$ satisfying $\sum_{\alpha=1}^s i_\alpha + \sum_{\beta=1}^r j_\beta = 0$, 
\begin{align}
    \state{i_1\cdots i_s j_1\cdots j_r}f - C\state{i_1\cdots i_s}\state{j_1'\cdots j_N'}f = 0,\label{reduction}
\end{align}
here $N = \left|\left(\sum_{\beta=1}^r j_\beta \right)\right|, j_1' = \cdots = j_N'=\text{sgn}\left(\sum_{\beta=1}^r j_\beta \right)$. 
\end{prop}
%
\begin{proof} 
Without loss of generality, assume that $\sum_{\beta=1}^r j_\beta < 0$. Let $P$ be the number of $+$'s in $j_1,...,j_r$. We perform induction on $P$. If $P=0$, then $N = r$, and $j_{k+1}= \cdots = j_r = -$. We simply take $C=1$ and $j_1' = \cdots = j_N' = -$. (\ref{reduction}) holds.  

Assume the conclusion holds when the number of +'s is strictly less than $P$. In case $j_r = -$, we further assume that $j_r=\cdots = j_{\gamma+1} = -, j_\gamma = +$ for some $\gamma\in [1, r]$. Using the same computation as in Proposition 5.5, we see that
\begin{align*}
    \state{i_1\cdots i_s j_{1}\cdots j_r}f & = \state{i_1\cdots  i_s j_{1}\cdots j_{\gamma-1} + - \cdots -}f\\
    & = \state{i_1\cdots i_s j_{k+1}\cdots j_{t-1} - \cdots -}\state{+-}f\\
    & \quad + K(r-\gamma-1)(r-\gamma) \state{i_1\cdots  i_s j_{1}\cdots j_{\gamma-1}\widehat{+-} - \cdots -}\\
    & = (-\lambda + K(r-\gamma-1)(r-\gamma))\state{i_1\cdots i_s j_{1}\cdots j_{\gamma-1}\widehat{+-} - \cdots -}.
\end{align*}
Here we use the hat notation to indicated the removed terms. Since $\state{i_1\cdots i_s j_{1}\cdots j_{\gamma-1}\widehat{+-} - \cdots -}$ contains one less $+$ and one less $-$ compared to the original $\state{i_1\cdots i_s j_{1}\cdots j_{\gamma-1}+- - \cdots -}$, the induction hypothesis gives a constant $C^{(0)} = C^{(0)}_{j_{1}...\widehat{j_\gamma}\widehat{j_{\gamma+1}}...j_r}(\lambda)$ depending only on $j_{1}, ..., j_{\gamma-1}, j_{\gamma+2}, ..., j_r$, such that 
$$\state{i_1\cdots i_s j_{1}\cdots j_{\gamma-1}\widehat{+-} - \cdots -} = C^{(0)}\state{i_1\cdots i_s}\state{j'_1\cdots j'_N},$$
with $j_1'= \cdots = j_N' = -$. 
Thus (\ref{reduction}) holds with $C= (-\lambda + K(r-t-1)(r-t))C_{j_1\cdots j_{t-1}j_{t+2}...j_r}^{(0)}(\lambda)$. 

In case $j_r=+$, there exists some $\gamma$ such that $\sum_{\beta=\gamma}^r j_\beta = 0$. Then 
$$\state{i_1\cdots i_s j_1 \cdots j_r} = \state{i_1 \cdots i_s j_1 \cdots j_{\gamma-1}}\state{j_\gamma\cdots j_r}f = C^{(1)}\state{i_1 \cdots i_s j_1 \cdots j_{\gamma-1}}f$$
for some constant $C^{(1)}= C^{(1)}_{j_{\gamma \cdots j_r}}(\lambda)$ depending only on $j_\gamma, ..., j_r$ and $\lambda$. By induction hypothesis, 
$\state{i_1 \cdots i_s j_1 \cdots j_{\gamma-1}}f = C^{(2)}\state{i_1 \cdots i_s j_1'\cdots j_N'}f$ for some $C^{(2)}=C^{(2)}_{j_{\gamma \cdots j_r}}(\lambda)$ depending only on $j_1, ..., j_{\gamma-1}$ and $\lambda$. Thus (\ref{reduction}) holds with $C(\lambda) = C^{(1)}= C^{(1)}_{j_{\gamma \cdots j_r}}(\lambda)C^{(2)}_{j_{\gamma \cdots j_r}}(\lambda)$. 
%
\end{proof}

\begin{rema}
In natural language, we will describe conclusion of Proposition \ref{5-5} as expressing $\state{i_1\cdots i_s j_1\cdots j_r}f$ uniformly as $C\state{i_1\cdots i_sj_1'\cdots j_N'}$ with respect to $i_1, ..., i_s$. This turns out to be convenient when we study the general criterion for an irreducible $V(l, \one)$-module to be grading-restricted later in this paper. 
\end{rema}

\begin{thm}\label{5-10}
Let $j_1, ..., j_r\in \{+, -\}$ satisfying $\sum_{p=1}^r j_p = 0$, $t_1, ..., t_k\in \Z_+$. Set $N =\left|\sum_{p=k+1}^r j_p\right|$. Then there exists a constant $C=C_{j_{k+1},..., j_r}(\lambda)$ depending only on $j_{k+1}, \cdots j_r$ and the eigenvalue $\lambda$, such that 
\begin{align}
    \state{j_1(-t_1) \cdots j_k(-t_k) j_{k+1}(0) \cdots j_r(0)}f - C(\lambda) \state{j_1(-t_1)\cdots j_k(-t_k) j'_{1}(0) \cdots j'_{N}(0)}f\label{submodgen}
\end{align}
generates a proper submodule. Here $j_1' = \cdots = j_{N}' = \text{sgn}\left(\sum_{p=k+1}^r j_p\right)$.
\end{thm}

\begin{proof}
Without loss of generality, assume $\sum_{p=k+1}^r j_p < 0$. This means among $j_{k+1}, ..., j_r$, there are more $-$'s than $+$'s, with $N$ being the difference. Now consider the lowest weight projection of $Y(v,x)w$, with $v=\state{i_1(-m_1)\cdots i_s(-m_s)}\one, w = \state{j_1(-t_1)\cdots j_k(-t_k) j_{k+1}(0)\cdots j_r(0)}f$. From Proposition \ref{proj0-prop}, the projection is a sum of elements of the form 
\begin{align*}
\prod_{\substack{1\leq p \leq s\\
p\neq c_1, ..., c_k}}(-1)^{m_p - 1}&\cdot \prod_{p=1}^k \frac{(-t_{\sigma(p)}-1) \cdots (-t_{\sigma(p)}-m_{c_p}+1)}{(m_{c_p}-1)!} \cdot \prod_{p=1}^k t_{\sigma(p)}l \langle i_{c_p}, j_{\sigma(p)} \rangle\\
& \cdot \state{i_1(0) \cdots \widehat{i_{c_k}(0)} \cdots \widehat{i_{c_1}(0)} \cdots i_s(0) j_{k+1}(0)\cdots j_r(0)}f
\end{align*}
From Proposition \ref{5-5}, we see that for every fixed $c_1, ..., c_k$ and $\sigma\in S_k$, 
\begin{align*}
\prod_{\substack{1\leq p \leq s\\
p\neq c_1, ..., c_k}}(-1)^{m_p - 1}&\cdot \prod_{p=1}^k \frac{(-t_{\sigma(p)}-1) \cdots (-t_{\sigma(p)}-m_{c_p}+1)}{(m_{c_p}-1)!} \cdot \prod_{p=1}^k t_{\sigma(p)}l \langle i_{c_p}, j_{\sigma(p)} \rangle\\
& \cdot \big{(} \state{i_1(0) \cdots \widehat{i_{c_k}(0)} \cdots \widehat{i_{c_1}(0)} \cdots i_s(0) j_{k+1}(0)\cdots j_r(0)}f \\
& \quad - \state{i_1(0) \cdots \widehat{i_{c_k}(0)} \cdots \widehat{i_{c_1}(0)} \cdots i_s(0)}\state{j_1'(0) \cdots j_N'(0)} f\big{)} = 0.
\end{align*}
with respect to any choice of $i_1, ..., i_s$ and $m_1, ..., m_s$. Thus for every homogeneous $v\in V$ and $w$ given by (\ref{submodgen}), the lowest weight projection of $Y(v, x)w$ does not contain $f$. Thus (\ref{submodgen}) generates a proper submodule. 
\end{proof}

\begin{rema}
From the proof, it is clear that the constant $C$ is a polynomial in $\lambda$ whose roots depend on the positioning of $+$ and $-$ in $j_{k+1},..., j_r$ and are contained in the set (\ref{setofroot}). So the constant $C$ is nonzero for generic eigenvalues $\lambda$. On the other hand, if the eigenvalue $\lambda$ happens to make $C=0$, then the element $\state{j_1(-t_1) \cdots j_k(-t_k) j_{k+1}(0) \cdots j_r(0)}f$ itself generates a submodule that can be quotiented out. 
\end{rema}

\begin{thm}\label{5-12}
Let $V_U^{(1)}(l, f)$ be the quotient of $V_U(l, f)$ by the submodule generated by elements of the form (\ref{submodgen}). 
\begin{enumerate}
    \item The set 
$$\left\{h_{i_1}(-t_1) \cdots h_{i_k}(-t_k)h_{i_{k+1}}(0) \cdots h_{i_r}(0)f:
\begin{aligned}
&r \geq 0, t_1, ..., t_k > 0\\
&\#\{j:i_j=+\} = \#\{j: i_j=-\}\\
&i_{k+1} = \cdots = i_r \in \{+, -\}
\end{aligned}\right\}$$
forms a basis for $V_U^{(1)}(l, f)$.
\item $V_U^{(1)}(l, f)$ is grading-restricted. The graded dimension of $V_U^{(1)}(l, f)$ is 
$$1 + \sum_{n=1}^\infty 2\cdot 3^{n-1}q^n.$$
\end{enumerate} 
\end{thm}

\begin{proof}
Let $C$ be the constant in  (\ref{submodgen}). In case $C\neq 0$, the relation (\ref{submodgen}) allows us to identify $h_{i_1}(-t_1)\cdots h_{i_k}(-t_k)h_{i_{k+1}}(0)f \cdots h_{i_r}(0)$ in (\ref{spanset-dim2}) with either $h_{i_1}(-t_1) \cdots h_{i_k}(-t_k)h_{+}(0) \cdots h_{+}(0) f$ or $h_{i_1}(-t_1) \cdots h_{i_k}(-t_k)h_{-}(0) \cdots h_{-}(0) f$. The former happens when $\sum_{p=k+1}^r i_p > 0$, with the number of $+$'s being $N = \left|\sum_{p=k+1}^r i_p\right|$. The latter happens when $\sum_{p=k+1}^r i_p < 0$, with the number of $-$'s being $N$ as well. In case $C=0$, $h_{i_1}(-t_1)\cdots h_{i_k}(-t_k)h_{i_{k+1}}(0)f \cdots h_{i_r}(0)$ is simply identified with zero and will also be quotiented out. In both cases, the first conclusion follows. 

To see $V_U^{(1)}(l,f)$ is grading-restricted, we fix $n>0$ and consider the homogeneous subspace of weight $\lambda+n$. Then $t_1 + \cdots + t_k = n$. Thus $k$ is at most $n$. Also, $i_1 + \cdots + i_k$ is bounded by $[-k, k]$. Since $i_{k+1} + \cdots + i_r = - (i_1 + \cdots + i_k)$, $r$ and $i_{k+1} = \cdots = i_r$, $r$ is bounded above by $2k$. Thus for fixed $n$, there are only finitely many choices for $k, t_1, ..., t_k, r, i_1, ..., i_r$.

To determine exactly how many choices, we first fix $k>0$. Then the choices for $t_1,..., t_k$ is $\binom {n-1}{k-1}$. We can choose $i_1, ..., i_k$ freely. But once $i_1, ..., i_k$ are fixed,  the choice for $r$ and $i_{k+1}, ..., i_r$ is unique. Thus the total number of choices is 
$$\sum_{k=1}^n \binom {n-1}{k-1} 2^k = 2\cdot \sum_{k=0}^{n-1}\binom{n-1}{k} 2^k = 2\cdot 3^{n-1}. $$
The second conclusion then follows. 
\end{proof} 

\begin{rema}
One can also assign the weight of $f$ to be other numbers. A natural choice for $\wt f$ is the eigenvalue $\lambda$ of $f$. In this case, the graded dimension will then be shifted by a factor $q^\lambda$. Since we are studying the modules generated by one eigenfunction at this moment, we will stick to the choice that $\wt f = 0$ for the module $V_U(l, f)$ for convenience. 
\end{rema}

\begin{rema}
A more conceptual way to obtain $V_U^{(1)}(l,f)$ is to start from the $T(E)$-module $(T(E)\otimes_{\Pi(T(E))} \C f) /N$, where $N$ is the $T(E)$-submodule generated by 
\begin{align}
    (X\otimes Y \otimes Z_1 \otimes \cdots \otimes Z_n) f - (Y\otimes X \otimes Z_1 \otimes \cdots \otimes Z_n) f 
    + \sum_{i=1}^n (Z_1 \otimes \cdots \otimes R(X,Y)Z_i \otimes \cdots \otimes Z_n)f
\end{align}
and form the induced module $T(\widehat{E_p}_-)\otimes ((T(E)\otimes_{\Pi(T(E))} \C f) /N)$. The $V(l, \one)$-submodule generated by $1\otimes (1\otimes f)$ is indeed isomorphic to $V_U^{(1)}(l, f)$. Conceptually, the submodule generated by $f$ amounts to requires nonparallel tensors in $T(E_p)$ to satisfy similar relations as the covariant derivatives on $C^\infty(U)$. But we are not going to actually act these nonparallel tensors on $C^\infty(U)$. 
\end{rema}

\subsection{Second quotient} 

\begin{thm}\label{5-13}
Let $j_1, ..., j_r\in \{+,-\}$ satisfy $\sum_{p=1}^r j_p = 0$. Let $t_1,..., t_k \in \Z_+$. Then for every $\tau \in S_k$, 
\begin{align}
\state{j_1(-t_1)\cdots j_k(-t_k) j_{k+1}(0) \cdots j_r(0)}f - \state{j_{\tau(1)}(-t_{\tau(1)}) \cdots j_{\tau(k)}(-t_{\tau(k)})j_{k+1}(0)\cdots j_r(0)}f
\label{submodgen-2}
\end{align}
generates a proper submodule. 
\end{thm}

\begin{proof}
We apply Proposition \ref{proj0-prop} to $w=\state{j_{\tau(1)}(-t_{\tau(1)}) \cdots j_{\tau(k)}(-t_{\tau(k)})j_{k+1}(0) \cdots j_r(0)} f$. Then (\ref{proj0-formula}) becomes
\begin{align*}
\sum_{1\leq c_k < \cdots < c_1 \leq s} \sum_{\sigma\in S_k}
\prod_{\substack{1\leq p \leq s\\
p\neq c_1, ..., c_k}}(-1)^{m_p - 1}&\cdot \prod_{p=1}^k \frac{(-t_{\sigma(\tau(p))}-1) \cdots (-t_{\sigma(\tau(p))}-m_{c_p}+1)}{(m_{c_p}-1)!} \cdot \prod_{p=1}^k t_{\sigma(\tau(p))}l \langle i_{c_p}, j_{\sigma(\tau(p))} \rangle\nonumber \\
& \cdot \state{i_1(0) \cdots \widehat{i_{c_k}(0)} \cdots \widehat{i_{c_1}(0)} \cdots i_s(0) j_{k+1}(0)\cdots j_r(0)}f.
\end{align*}
Since the summation is done over all
$\sigma\in S_k$, shifting every $\sigma$ by $\tau$ does not change the sum. Thus the weight-$\lambda$ projection of $w = \state{j_{\tau(1)}(-t_{\tau(1)}) \cdots j_{\tau(k)}(-t_{\tau(k)})j_{k+1}(0) \cdots j_r(0)}$ with every $v\in V$ coincides with that of $\state{j_{1}(-t_{1}) \cdots j_{k}(-t_{k})j_{k+1}(0) \cdots j_r(0)}$. The conclusion then follows. 
\end{proof}

\begin{thm}\label{5-16}
Let $V_U^{(2)}(l, f)$ be the quotient of $V_U^{(1)}(l, f)$ by the submodule generated by elements of the form (\ref{submodgen-2}). Then the union of the following three sets
$$\left\{h_{+}(-t_1) \cdots h_{+}(-t_{r/2}) h_{-}(-t_{r/2+1}) \cdots h_{-}(-t_r)f:
 \begin{aligned} & r\geq 0 \text{ even, }t_1\geq \cdots \geq t_{r/2}\geq 1, \\ & t_{r/2+1} \geq \cdots \geq t_{r}\geq 1
\end{aligned}\right\},$$

$$\left\{
h_{+}(-t_1)\cdots h_{+}(-t_{r/2}) h_{-}(-t_{r/2+1}) \cdots h_{-}(-t_k)h_{-}(0)\cdots h_{-}(0)f: 
\begin{aligned}
& r\geq 0 \text{ even, } r/2 \leq k \leq r-1 \\ 
& t_1\geq \cdots \geq t_{r/2}\geq 1,\\ & t_{r/2+1} \geq \cdots \geq t_{k}\geq 1
\end{aligned}\right\},$$

$$\left\{
h_{-}(-t_1)\cdots h_{-}(-t_{r/2}) h_{+}(-t_{r/2+1}) \cdots h_{+}(-t_k)h_{+}(0)\cdots h_{+}(0)f: 
\begin{aligned}
& r\geq 0 \text{ even, } r/2 \leq k \leq r-1 \\ 
& t_1\geq \cdots \geq t_{r/2}\geq 1,\\ & t_{r/2+1} \geq \cdots \geq t_{k}\geq 1
\end{aligned}\right\},$$
forms a basis of $V_U^{(2)}(l, f)$. 
\end{thm}

\begin{proof}
The relation (\ref{submodgen-2}) allows us to permute all the negative modes arbitrarily in the quotient. Note that $\#\{j: i_j=+\}=\#\{j: i_j = -\}$. In case there is no zero modes, all the $h_+$ can be placed to the front and all the $h_-$ to the rear, forming two groups. In each group, higher weights can then be arranged to the front, lowers in the rear. In case there exists zero modes, by Theorem \ref{5-12}, either all zero modes are $h_+(0)$, or are $h_-(0)$. We arrange the terms similarly according to the choices of zero modes. 
\end{proof}

\begin{rema}
It is not difficult to see that the graded-dimension of $V_U^{(2)}(l, f)$ is 
$$1 + \sum_{n=1}^\infty\left[\sum_{\substack{r \text{ even,}\\1\leq \frac r 2 \leq \frac n 2}}\sum_{m=r/2}^{n-r/2} p_{r/2}(m)p_{r/2}(n-m) + 2\sum_{\substack{r \text{ even,}\\
1\leq \frac r 2 \leq n}}\left(p_{r/2}(n) + \sum_{m=1}^{n-1}\sum_{k=1}^{r/2-1} p_{r/2}(m)p_k(n-m)\right)\right]q^n,$$
where $p_{k}(n)$ is the number of unordered partitions of $n$ into exactly $k$ parts, i.e., the number of $\lambda_1, ..., \lambda_k$ such that $\lambda_1 \geq \cdots \geq \lambda_k\geq 1, \lambda_1 + \cdots +\lambda_k = n$. Whether or not this series can be further simplified in terms of some special functions remains a problem.
\end{rema}

We will proceed to show that $V_U^{(2)}(l,f)$ is irreducible if the eigenvalue of $f$ is generic (recall Remark \ref{genspe}). The proof will need the following fact in the polynomial algebra. 

\begin{lemma}\label{lin-indep}
\begin{enumerate}
    \item The following set of polynomials \begin{align}\left\{\sum_{\sigma\in S_k} \binom{t_{\sigma(1)}+x_1 - 1}{t_{\sigma(1)}} \cdots  \binom{t_{\sigma(n)}+x_n - 1}{t_{\sigma(n)}}: t_1 \geq \cdots t_n \geq 1\right\}\label{Formula-59}
    \end{align}
is a linearly independent subset in $\C[x_1, ..., x_n]$. 
    \item The following set of polynomials
    $$\left\{\begin{aligned}
    & \sum_{\sigma\in S_n} \binom{t_{\sigma(1)}+x_1 - 1}{t_{\sigma(1)}} \cdots  \binom{t_{\sigma(n)}+x_n - 1}{t_{\sigma(n)}}\\ 
    & \cdot \sum_{\tau\in S_m} \binom{s_{\tau(1)} + y_1 - 1}{s_{\tau(1)}}\cdots \binom{s_{\tau(m)} + y_1 - 1}{s_{\tau(m)}}
    \end{aligned}: t_1 \geq \cdots t_n \geq 1, s_1 \geq \cdots \geq s_m \geq 1\right\}$$
\end{enumerate}
is a linearly independent subset in $\C[x_1, ..., x_n, y_1, ..., y_n]$. 
\end{lemma}

\begin{proof}
For (1), we consider the linear map from $\C[x_1, ..., x_n]$ to itself defined by
$$x_1^{t_1} \cdots x_n^{t_n} \mapsto \binom{t_1 + x_1 - 1}{t_1} \cdots \binom{t_n + x_n - 1}{t_n}. $$
Since the highest degree term of the image is $x_1^{t_1}\cdots x_n^{t_n}$, it is clear that the matrix of the linear map with respect to the basis $x_1^{t_1} \cdots x_n^{t_n}$ is upper-triangular with non-vanishing diagonal entries. Thus the linear map is invertible. So the set (\ref{Formula-59}) is linearly independent if and only if 
$$\left\{\sum_{\sigma\in S_n} x_1^{t_{\sigma(1)}} \cdots x_n^{t_{\sigma(n)}}: t_1 \geq \cdots \geq t_n \geq 1\right\}$$
is linearly independent, which holds from the linear independence of the set of monomial symmetric polynomials (modified by a positive integer).

For (2), note that $\C[x_1, ..., x_n, y_1, ..., y_m] = \C[x_1, ..., x_n]\otimes \C[y_1, ..., y_m]$. The conclusion then follows from the following general fact: if $S$ is a linearly independent subset in $V$, $T$ is a linearly independent subset in $W$, then $$\{s\otimes t: s\in S, t\in T\}$$ 
is a linearly independent subset in $V\otimes W$. 
\end{proof}

\begin{thm}\label{5-18}
Let $\lambda$ be a generic eigenvalue, then $V_U^{(2)}(l, f)$ is irreducible. 
\end{thm}

\begin{proof}
We first consider the case when $w\in V_U^{(2)}(l, f)$ is homogeneous of weight $\omega$. We show that if the lowest weight projection of $Y(v, x)w$ is zero for every $v\in V$, then $w=0$. 

Set 
$$w = \sum a_{i_1\cdots i_r}^{t_1\cdots t_k} \state{i_1(-t_1)\cdots i_{r/2}(-t_{r/2}) i_{r/2+1}(-t_{r/2+1})\cdots i_k (-t_k)i_{k+1}(0)\cdots i_r(0)}f, $$
here the sum is over all possible choices of $r\geq 0$ even, $i_1, .., i_r\in \{+,-\}, r/2\leq k \leq r$, and $t_1, ..., t_k$ such that $t_1+\cdots + t_k = \omega$, $t_1 \geq \cdots \geq t_{r/2}, t_{r/2+1}\geq \cdots \geq t_k$. We proceed by induction of $r$ that every $a_{i_1\cdots i_r}^{t_1\cdots t_k} = 0$

From the first part of Proposition \ref{proj0-prop}, if we choose $v=h_+(-m_1)h_-(-m_2)\one$, then except for those terms with $k\leq 2$, all other terms in the lowest weight projection of $Y(v, x)w$ are zero. So it suffices to consider only the following part in $w$: 
$$a_{+-}^{\omega}h_+(-\omega) h_-(0)f + a_{-+}^{\omega}h_-(-\omega)h_+(0)f + \sum_{t= 1}^{\omega-1}a_{+-}^{t,\omega-t}h_+(-t)h_-(-\omega+t)f. $$
By assumption and (\ref{proj0-formula-2}), 
\begin{align*}
    & a_{+-}^{\omega}(-1)^{m_1+m_2-2}\binom{\omega + m_2 -1}{\omega}\omega l\state{+-}f+a_{-+}^\omega (-1)^{m_1+m_2-2}\binom{\omega+m_1-1}{\omega}\omega l\state{-+}f \\
    & \qquad + \sum_{t=1}^{\omega-1} a_{+-}^{t,\omega-t} (-1)^{m_1+m_2-2} \binom{t+m_2 - 1}{t}\binom{\omega-t+m_1-1}{\omega-t}t(\omega-t)l^2 f = 0.
\end{align*}
Note that the left-hand-side can be regarded as a polynomial in $m_1, m_2$ after we remove $(-1)^{m_1+m_2}$. Since the equality holds for every $m_1, m_2\in \Z_+$, thus left-hand-side, as a polynomial in $m_1, m_2$, has to be zero. Since polynomials involving different variables are linearly independent, it follows that
$$a_{+-}^\omega \omega l \state{+-}f = a_{-+}^\omega \omega l \state{-+} f = 0.$$
Using Lemma \ref{lin-indep} (with $n=1,m=1$), we also see that 
\begin{align*}
  a_{+-}^{t,\omega-t} t(\omega-t) l^2 f = 0, t = 1,..., \omega -1. 
\end{align*}
Since $l\neq 0, \omega \in \Z_+$, and $\lambda$ is generic, it forces that 
$$a_{+-}^\omega = a_{-+}^\omega =0,  a_{+-}^{t, \omega-t} = 0, t = 1, ..., \omega -1. $$
This finishes the proof of the base case $r=2$. 

Now assume that $a_{j_1...j_r}^{t_1...t_k} = 0$ for smaller $r$. Then we apply $h_{i_1}(-m_1)\cdots h_{i_{r}}(-m_{r})\one$ to $w$. 
By (\ref{proj0-formula-2}), 
\begin{align}\label{irred-formula-1}
    \sum a_{j_1 \cdots j_r}^{t_1 \cdots t_k} \sum_{1\leq c_k < \cdots < c_1 \leq r} \prod_{p=1}^k(-1)^{m_p-1} &\left(\sum_{\sigma\in S_k} \prod_{p=1}\binom{t_{\sigma(p)}+m_{c_p} -1}{t_{\sigma(p)}}\prod_{p=1}^k t_{\sigma(p)}l\langle i_{c_p}, j_{\sigma(p)}\rangle\right)\nonumber \\
    & \cdot \state{i_1(0)\cdots \widehat{i_{c_k}(0)} \cdots \widehat{i_{c_1}(0)}\cdots i_r(0) j_{k+1}(0)\cdots j_r(0)} = 0,
\end{align}
where the sum is done over all possible choices of $j_1, ..., j_r, t_1, ..., t_k$ satisfying the conditions specified in 
Theorem \ref{5-16}, and $t_1+\cdots + t_k = \omega$. We again view the left-hand-side as a polynomial in $m_1, ..., m_r$. For different choices of $c_1, ..., c_k$, the corresponding polynomials are in the set (\ref{Formula-59}) while involve different variables and thus are linearly independent. So for every fixed $c_1, ..., c_k$ satisfying $1\leq c_1 < \cdots < c_k \leq r$, 
\begin{align}
    \sum a_{j_1 \cdots j_r}^{t_1 \cdots t_k}  \prod_{p=1}^k(-1)^{m_p-1} &\left(\sum_{\sigma\in S_k} \prod_{p=1}\binom{t_{\sigma(p)}+m_{c_p} -1}{t_{\sigma(p)}}\prod_{p=1}^k t_{\sigma(p)}l\langle i_{c_p}, j_{\sigma(p)}\rangle\right)\nonumber \\
    & \cdot \state{i_1(0)\cdots \widehat{i_{c_k}(0)} \cdots \widehat{i_{c_1}(0)}\cdots i_r(0) j_{k+1}(0)\cdots j_r(0)} f= 0.
\end{align}
For every fixed $j_1, ..., j_r$ and $c_1, ..., c_k$, we choose $i_1, ..., i_r$ so that $i_{c_1} = -j_1, ..., i_{c_k} = -j_k$. Thus $$\prod_{p=1}^k t_{p} l \langle i_{c_p}, j_{p}\rangle = 2^k l^k t_1 \cdots t_k. $$ 
Since $j_1 = \cdots = j_{r/2} = - j_{r/2+1}= \cdots = -j_{r}$, we know that 
$$\prod_{p=1}^k t_{\sigma(p)} l \langle i_{c_p}, j_{\sigma(p)}\rangle =\left\{\begin{array}{ll} 2^k l^k t_1 \cdots t_k & \text{ if }\sigma\in S_{r/2}\times S_{k-r/2}, \\
0 & \text{ otherwise.}
\end{array}\right.$$ 
Here $S_{r/2}$ is the permutation group of $\{1, ..., r/2\}$, $S_{k-r/2}$ is the permutation group of $\{r/2+1, ..., k\}$. 
Thus the summation over $S_n$ reduces to the summation over the subgroup $S_{r/2}\times S_{k-r/2}$. Rewrite $\sigma$ as $\sigma \cdot \tau$ in $S_{r/2}\times S_{k-r/2}$, the sum becomes
\begin{align}
\sum a_{j_1\cdots j_r}^{t_1\cdots t_k} \prod_{p=1}^k (-1)^{m_p-1} & \left(\sum_{\sigma \in S_{r/2}}\prod_{p=1}^{r/2} \binom{t_{\sigma(p)}+m_{c_p}-1}{t_{\sigma(p)}}\right)\left( \sum_{\tau\in S_{k-r/2}}\prod_{p=r/2+1}^k \binom{t_{\tau(p)}+m_{c_p}-1}{t_{\tau(p)}}\right)\nonumber \\
& \cdot 2^k l^k t_1 \cdots t_k \state{i_1(0)\cdots \widehat{i_{c_k}(0)} \cdots \widehat{i_{c_1}(0)}\cdots i_r(0) j_{k+1}(0)\cdots j_r(0)}f = 0.
\end{align}
Since $k\geq r/2$, the outer sum here is over only one choice of $j_1, ..., j_r$ we fixed above, and over all possible $t_1, ..., t_k$ satisfying $t_1 \geq \cdots \geq t_{r/2}\geq 1, t_{r/2}\geq \cdots \geq t_k \geq 1$. The second part of Lemma \ref{lin-indep} tells us that for every fixed $t_1, ..., t_k$, \begin{align}\label{irred-formula-2}
    a_{j_1\cdots j_r}^{t_1\cdots t_k} \prod_{p=1}^k (-1)^{m_p-1}\cdot  2^k l^k t_1\cdots t_k\cdot  \state{i_1 \cdots \widehat{i_{c_k}} \cdots \widehat{i_{c_1}} \cdots i_r j_{k+1} \cdots j_r} f = 0.
\end{align}
Since $\lambda$ is generic, $l\neq 0, t_1, ..., t_k \neq 0$, we thus have
$$a_{j_1\cdots j_r}^{t_1\cdots t_k}=0$$
for every fixed $j_1, ..., j_r$ and $t_1, ..., t_k$ satisfying the conditions in Theorem \ref{5-16}. 
\end{proof}


\subsection{Third quotient} We have shown that $V_U^{(2)}(l, f)$ is irreducible if the eigenvalue $\lambda$ is generic. For $\lambda=p(p-1)K$, we have the following conclusion,

\begin{thm}
Let $\lambda = p(p-1)K$ be the eigenvalue of $f$ for some $p\in\Z_+$. Then for every even number $r\geq k + p$, every $t_1 \geq \cdots \geq t_{r/2}, t_{r/2+1}\geq \cdots \geq t_{k}$, the element 
\begin{align}
    \state{j_1(-t_1) \cdots j_{r/2}(-t_{r/2}) j_{r/2+1}(t_{-r/2+1}) \cdots j_{k}(-t_k)j_{k+1}(0)\cdots j_r(0)}f \label{submodgen-3}
\end{align}
generates a proper submodule in $V_U^{(2)}(l, f)$. 
\end{thm}

\begin{proof}
Note that $j_{k+1} = \cdots = j_r$. Thus in  (\ref{proj0-formula}), the term 
$$\state{i_1(0)\cdots \widehat{i_{c_k}}(0) \cdots \widehat{i_{c_1}}(0) \cdots i_s(0) j_{k+1}(0) \cdots j_r(0)}f$$
contains at least $p$ consecutive $+$ or $-$. From Proposition \ref{4-4} Part (2), $\lambda = p(p-1)K$ annihilates (\ref{proj0-formula}). The conclusion then follows. 
\end{proof}

\begin{thm}\label{5-20}
Let $\lambda = p(p-1)K$ be the eigenvalue of $f$ for some $p\in \Z_+$. Let $V_U^{(3)}(l, f)$ be the quotient of $V_U^{(2)}(l, f)$ by the submodule generated by elements of the form (\ref{submodgen-3}). Then the the union of the following three sets
$$\left\{h_{+}(-t_1) \cdots h_{+}(-t_{r/2}) h_{-}(-t_{r/2+1}) \cdots h_{-}(-t_r)f:
 \begin{aligned} & r\geq 0 \text{ even, }t_1\geq \cdots \geq t_{r/2}\geq 1, \\ & t_{r/2+1} \geq \cdots \geq t_{r}\geq 1
\end{aligned}\right\},$$

$$\left\{
h_{+}(-t_1)\cdots h_{+}(-t_{r/2}) h_{-}(-t_{r/2+1}) \cdots h_{-}(-t_k)h_{-}(0)\cdots h_{-}(0)f: 
\begin{aligned}
& r\geq 0 \text{ even; }\\ & \max(r-p+1,r/2) \leq k \leq r-1; \\ 
& t_1\geq \cdots \geq t_{r/2}\geq 1,\\ & t_{r/2+1} \geq \cdots \geq t_{k}\geq 1
\end{aligned}\right\},$$

$$\left\{
h_{-}(-t_1)\cdots h_{-}(-t_{r/2}) h_{+}(-t_{r/2+1}) \cdots h_{+}(-t_k)h_{+}(0)\cdots h_{+}(0)f: 
\begin{aligned}
& r\geq 0 \text{ even; }\\ & \max(r-p+1,r/2) \leq k \leq r-1; \\ 
& t_1\geq \cdots \geq t_{r/2}\geq 1,\\ & t_{r/2+1} \geq \cdots \geq t_{k}\geq 1
\end{aligned}\right\},$$
forms a basis of $V_U^{(3)}(l, f)$. In other words, there are at most $p-1$ consecutive $h_-(0)$ (resp., $h_-(0)$) at the rear of the second (resp., the third) set of basis. 
\end{thm}

\begin{thm}
Let $\lambda = p(p-1)K$ be the eigenvalue of $f$. Then $V_U^{(3)}(l, f)$ is irreducible. 
\end{thm}

\begin{proof}
Observe that for every basis element in Theorem \ref{5-20}, there exists $i_1, ..., i_s$ and a choice of $c_1, ..., c_k$ such that 
$$\state{i_1(0)\cdots \widehat{i_{c_k}(0)} \cdots \widehat{i_{c_1}(0)} \cdots i_s(0) j_{k+1}(0) \cdots j_r(0)}f \neq 0.$$
With this observation, we can formulate an argument similarly as Theorem \ref{5-18}. Details shall not be repeated here. 
\end{proof}

\begin{rema}
In case $f$ is a global eigenfunction over a compact space form of negative sectional curvature, then the eigenvalue of $f$ must be a positive real number and will never coincide with $p(p-1)K$. Thus, $V_M^{(2)}(l, f)$ is irreducible. 
\end{rema}

\begin{rema}
In case $M = \mathbb{S}^2$, the two-dimensional unit sphere in $\R^3$, then $K=1$. It is also well known that if $f$ is an eigenfunction over $M$, then eigenvalue of $f$ is $p(p-1)$ for some $p\in \Z_+$. Thus $V_M^{(2)}(l, f)$ is not irreducible. One has to take the third quotient to obtain the irreducible $V_M^{(3)}(l, f)$.
\end{rema}

\begin{nota}
One easily sees from Theorem \ref{isom-reln} that if $f_1, f_2$ are two functions with the same eigenvalue $\lambda$, then the irreducible quotient of $V_U(l, f_1)$ and $V_U(l, f_2)$ are isomorphic. 
\end{nota}


\section{Some Classification Results on Lowest Weight Modules}

In this section, we will set up a correspondence between irreducible $V(l, \one)$-modules and irreducible $\Pi(T(E))$-modules. Then we introduce a geometrically interesting condition on the $\Pi(T(E))$-modules, called the covariant derivative condition. We will then classify all irreducible $V(l, \one)$-modules and all $V(l, \one)$-modules of finite length, whose lowest weight subspaces satisfy this covariant derivative condition and generate the whole module. 

For convenience, we will use $V$ to denote the MOSVA $V(l, \one)$ hereafter. 

\subsection{Lowest weight $V$-modules}

\begin{defn}
A $V$-module $W$ is a lowest weight $V$-module if there exists some $\mu\in \C$, such that $W=\coprod_{n\in \N} W_{[\mu+n]}$, and $W$ is generated by the lowest weight space $W_{[\mu]}$. 
\end{defn}

Obviously $W_{[\mu]}$ is a $\Pi(T(E))$-module. On the other hand, given a $\Pi(T(E))$-module $\Phi$, we consider the $V$-submodule $W^{(0)}(\Phi, [\mu])$ of the induced $T(\widehat{E_p}_-)$-module 
$$T(\widehat{E_p}_-)\otimes T(E) \otimes_{\Pi(T(E))} \Phi. $$
where we assign the subspace $1\otimes 1\otimes \Phi$ the weight $\mu\in \C$. 

\begin{prop}
$W^{(0)}(\Phi, [\mu])$ is universal in the following sense that every lowest weight $V$-module with $\Phi$ as the lowest weight subspace is a quotient of $W^{(0)}(\Phi, [\mu])$. 
\end{prop}

\begin{proof}
It follows from Theorem \ref{4-10} that both $W^{(0)}(\Phi, [\mu])$ and $W$ are spanned by 
\begin{align*}
\left\{  h_{i_1}(-t_1) \cdots h_{i_k}(-t_k) h_{i_{k+1}}(0) \cdots h_{i_{r}}(0)f: \begin{aligned} & f\in \Phi; r\geq 0, k = 1, ..., r, t_1, .., t_k > 0; \\
&\#\{j\in [1,r]: i_j = +\} = \#\{j\in [1,r]: i_j = -\}; \\
&\forall p = k+1, ..., r, \\
& \#\{j\in [p,r]: i_j = + \}\neq \#\{j\in [p,r]: i_j = - \}.\\
\end{aligned}\right\}.`
\end{align*}
Note also that the set forms a basis for $W^{(0)}(\Phi, [\mu])$. The map sending the basis vectors of $W^{(0)}(\Phi, [\mu])$ to the spanning vectors of $W$ is obviously a homomorphism of $V$-modules. Then $W$ is the quotient of $W^{(0)}(\Phi, [\mu])$ by the kernel of the map.
\end{proof}

\begin{prop}\label{6-1}
Let $W$ be an irreducible $V$-module. Then $W$ is a lowest weight $V$-module. The lowest weight subspace of $W$ is an irreducible $\Pi(T(E))$-module. 
\end{prop}

\begin{proof}
Fix $m\in \C$ such that $W_{[m]}\neq 0$. Then from the axiom of $V$-modules, there exists a smallest number $\mu$ in the congruence class $ m + \Z \in \C/\Z$, such that $W_{[\mu]}\neq 0$. Since $W$ is irreducible, then from Theorem \ref{4-10}, we know that $W = \coprod_{n\in \N} W_{[\mu+n]}$. From the definition of the vertex operator $Y(v,x)$, we see that the lowest weight subspace $W_{[\mu]}$ is also a $\Pi(T(E))$-module. Since $W$ is irreducible, for any $w\in W_{[\mu]}$, the submodule generated by $w$ must also coincide with $W$. In particular, the $\Pi(T(E))$-submodule generated by any $w\in W_{[\mu]}$ coincides with $W_{[\mu]}$. Thus $W_{[\mu]}$ is irreducible. 
\end{proof}

\begin{prop}\label{6-2}
Let $\Phi$ be an irreducible $\Pi(T(E))$-module. Then for any $\mu\in \C$, there exists a unique irreducible $V$-module $W(\Phi, [\mu])= \coprod_{n\in \N} W_{[\mu+n]}$ such that the lowest weight space $W_{[\mu]}$ is $\Phi$. 
\end{prop}

\begin{proof}
Consider the universal lowest weight $V$-module $W^{(0)}(\Phi,[\mu])$. It is clear that a $V$-submodule of $W^{(0)}(\Phi,[\mu])$ is proper if and only if it intersects the lowest weight subspace trivially. So the sum of two proper submodules stays proper. Thus there exists a unique maximal submodule. The quotient of $W^{(0)}(\Phi, [\mu])$ by this unique maximal submodule is thus the irreducible $V$-module satisfying the conditions. 
\end{proof}

\begin{rema}
It follows from the uniqueness in Proposition \ref{6-2} that there is a bijective correspondence between irreducible $\Pi(T(E))$-modules and the irreducible $V$-modules. 
\end{rema}

\begin{rema}\label{6-3}
If the irreducible $V$-module is grading-restricted, then its lowest weight subspace is certainly a finite-dimensional. But the converse does not necessarily hold. Indeed, using the lowest weight projection formula, one can show that an irreducible $V$-module is grading-restricted, if and only if the lowest weight subspace is finite-dimensional, and the action of $\Pi(T(E))$ satisfies the following technical condition: 

For every $w\in W_{[\mu]}$, every $\zeta\in \Z\setminus\{0\}$, there exists $R\in \Z_+$, such that for every $r>R$, $j_1, ..., j_r\in \{+,-\}$ satisfying $\sum_{p=1}^r j_p = \zeta$, there exists $C_{j'_1 \cdots j'_{r'}}$ for every $r'\leq R, j'_1, ..., j'_{r'}\in \{+, -\}$ satisfying $\sum_{q=1}^{r'} j'_{q'} = \zeta$, such that for every $i_1, ..., i_s$ satisfying $\sum_{p=1}^s i_p = -\zeta$, 
$$\state{i_1\cdots i_sj_1\cdots j_r}w = \sum C_{j'_1 \cdots j'_{r'}} \state{i_1 \cdots i_k j'_1 \cdots j'_{r'}}w,$$
where the sum is over all $r'\leq R$ and all choices of $j'_1, ..., j'_{r'}$ satisfying $\sum_{q=1}^{r'} j'_q = \zeta$. \\
In other words, $\state{i_1 \cdots i_k j_1 \cdots j_r} w$ can be expressed as a finite linear combination of terms whose tail length is bounded above, and the expression is uniform with respect to $i_1, ..., i_s$. 
\end{rema}

\begin{rema}
Though the technical condition in the previous remark reduces the problem of classifying irreducible grading-restricted $V$-modules to classifying irreducible finite-dimensional $\Pi(T(E))$-modules satisfying the technical condition, in practice it is not restrictive enough for us to work out an actual classification. We should emphasize that $\Pi(T(E))$, regarded as a subalgebra of $T(E_p)$, contains infinitely many generators that are algebraically independent. Thus to identify even a one-dimensional irreducible $\Pi(T(E))$-module requires a specification of infinitely many constants representing the actions of these infinitely many generators. Specifying all possible choices of these infinitely many constants satisfying the technical condition is highly nontrivial. 
\end{rema}

\subsection{Covariant derivative condition}

Instead of working on the technical necessary and sufficient condition in Remark \ref{6-3}, we will instead work on the following sufficient condition that is geometrically interesting. 

\begin{defn}
Let $N$ be a $\Pi(T(E))$-module. We say $N$ satisfies the covariant derivative condition if for every $\sum_{i}X_{(1)i}\otimes \cdots X_{(r)i}\in \Pi(T(E))$, 
\begin{align}
\sum_{i} \left(X_{(1)i} \otimes \cdots \otimes X_{(r-1)i}\otimes X_{(r)i} - X_{(1)i} \otimes \cdots \otimes X_{(r)i}\otimes X_{(r-1)i}\right) \label{CovDerReln-1}
\end{align}
act by zero, and for $i=1, ..., r-2,$
\begin{align}
    & \sum_{i} \left(X_{(1)i} \otimes \cdots \otimes X_{(j)i} \otimes X_{(j+1)i}\otimes  \cdots\otimes X_{(r)i} - X_{(1)i} \otimes \cdots \otimes X_{(j+1)i} \otimes X_{(j)i}\otimes  \cdots\otimes X_{(r)i}\right)\nonumber\\
    & \qquad + \sum_i \sum_{k=j+2}^r \left(X_{(1)i} \otimes \cdots \otimes X_{(j+2)i} \otimes \cdots R(X_{(j)i}, X_{(j+1)i})X_{(k)i}\otimes \cdots\otimes X_{(r)i} \right)\label{CovDerReln-2}
\end{align}
act by zero. 
\end{defn}


\begin{rema}
$\Pi(T(E))$-modules satisfying the covariant derivative condition automatically satisfies the condition in Remark \ref{6-3}. For every $w\in N$, and $\zeta\in \Z\setminus\{0\}$, $R$ can simply be chosen as $|\zeta|$. Then for every $r>|\zeta|$, and every $j_1, ..., j_r\in \{+,-\}$ satisfying $\sum_{p=1}^r j_p = \zeta$, there exists a constant $C$, such that for every $i_1, ..., i_s\in \{+, -\}$, 
$$\state{i_1\cdots i_s j_1 \cdots j_r} = C \state{i_1\cdots i_k} \state{\alpha}^R.$$
where $\alpha = \text{sgn}(\zeta)$ (cf. Theorem \ref{5-10}). 
\end{rema}

\begin{rema}
The covariant derivative condition is just one of the possible candidates for geometrically interesting conditions. Another option of such a condition comes from the action of parallel tensors on certain differential forms and shall be discussed in future work. 
\end{rema}

The covariant derivative condition can be used to express the actions of all generators of $\Pi(T(E))$ in terms of the action of $\state{+-}$. 

\begin{prop}\label{6-8}
Let $M$ be a finite-dimensional $\Pi(T(E))$-module satisfying the covariant derivative condition. Then for every $i_1, ..., i_r\in \{+, -\}$ with $\sum_{p=1}^r i_r = 0$, the action of $\state{i_1\cdots i_r}$ is uniquely determined by the action of $\state{+-}$, and can be expressed as a polynomial in $\state{+-}$. 
\end{prop}

\begin{proof}
In a nutshell, we can process the contraction of $+-$ pairs similarly as in Proposition \ref{4-4}. More elaborately, we show the conclusion by induction on $r$. For $r=2$, we see from (\ref{CovDerReln-1}) that
$$\state{-+}= \state{+-}.$$
Now assume the conclusion for all smaller $r$. Without loss of generality, let $i_r = i_{r-1} = \cdots = i_{k+1}=-, i_k = +$. Then from (\ref{CovDerReln-2}), 
\begin{align*}
    \state{i_1\cdots i_k i_{k+1} \cdots i_r} &= \state{i_1 \cdots i_k + -\cdots -} \\
    &= \state{i_1 \cdots i_k -+ \cdots -} - \sum_{j=k+2}^r \state{i_1\cdots i_k - \cdots-, R(+-)-, - \cdots -}\\
    &= \state{i_1 \cdots i_k -+ \cdots -} + 2K(r-k-1) \state{i_1\cdots i_k - \cdots-- \cdots -}.
\end{align*}
In the first term, the $+$ in is now moved to $(k+2)$-th position. The second term, by induction hypothesis, is a polynomial of $\state{+-}$. Repeating the process to move $+$ further right and handle the extra term similarly by induction hypothesis, until when we have $$\state{i_1\cdots i_k -\cdots -+-} = \state{i_1\cdots i_k-\cdots -}\state{+-}.$$
The induction hypothesis implies that first factor is a polynomial in $\state{+-}$, and the second factor is precisely $\state{+-}$. Thus we proved the conclusion for $r$. 
\end{proof}

\begin{prop}
Let $M$ be a finite-dimensional $\Pi(T(E))$-module. If $M$ is irreducible, then $M$ is one-dimensional. 
\end{prop}

\begin{proof}
Let $m\in M$ be an eigenvector for $\state{+-}$. Then from Proposition \ref{6-8}, for every $i_1, ..., i_r\in \{+,-\}$ with $\sum_{p=1}^r i_p = 0$, $m$ is an eigenvector for $\state{i_1\cdots i_r}$. Thus $\C m$ is a submodule of $M$ and has to coincide with $M$. 
\end{proof}


This result, combined with Proposition \ref{6-1}, \ref{6-2} and Theorem \ref{4-10}, yields the following classification theorem of irreducible $V$-modules satisfying the covariant derivative condition. 

\begin{thm}\label{6-9}
Let $W$ be an irreducible $V$-module such that the lowest weight subspace $W_{[\mu]}$ is a finite-dimensional $\Pi(T(E))$-module satisfying the covariant derivative condition. Then $W$ is isomorphic to the irreducible quotient of $V_U(l, f)[\mu]$ for some eigenfunction $f$. Here $V_U(l, f)[\mu]$ is the $V$-module generated by $f$ with the weight of $f$ specified as $\mu$. In other words, every irreducible $V$-module satisfying the covariant derivative condition is generated by an eigenfunction (locally over an open subset of the manifold $M$). 
\end{thm}

\begin{nota}
We will use the notation $W(\lambda,[\mu])$ to denote the irreducible quotient of $V_U(l, f)[\mu]$. 
\end{nota}

\subsection{$V$-modules of class $\mathscr{C}$: indecomposable case} 
\begin{defn}
A $V$-modules $W$ is of class $\mathscr{C}$, if it satisfies the following two conditions: 
\begin{enumerate}
    \item $W$ is a lowest weight module. The lowest weight subspace of $W$, as a  $\Pi(T(E))$-module, satisfies the covariant derivative condition. 
    \item $W$ admits a composition series 
    $$W = W_1 \supseteq W_2 \supseteq \cdots \supseteq W_n \supseteq W_{n+1} = 0$$
    of finite length, such that for $i = 1, ..., n$, the composition factor $W_i / W_{i+1}$ isomorphic to $W(\lambda_i, [\mu])$ for some $\lambda_i \in \C$ and some $\mu\in \C$.
\end{enumerate}
\end{defn}

To classify $V$-modules of class $\mathscr{C}$. We first classify the indecomposable modules with identical composition factors, i.e., $\lambda_1 = \cdots = \lambda_n = \lambda.$

\begin{prop}\label{6-11}
Let $\Phi$ be a $n$-dimensional $\Pi(T(E))$-module satisfying the covariant derivative condition. If $\Phi$ is indecomposable, then $\Phi$ is isomorphic to the $n$-dimensional $\Pi(T(E))$-module where the Jordan canonical form of $\state{+-}$ is a single Jordan block 
$$\begin{bmatrix}
\lambda & 0 &   \cdots & 0 & 0\\
1 & \lambda &   \ddots  & 0 & 0\\
\vdots & \ddots & \ddots & \ddots &\vdots \\ 
0 & 0 & \ddots & \lambda & 0 \\
0 & 0 & \cdots & 1 & \lambda
\end{bmatrix} $$
of dimension $n$. 
\end{prop}

\begin{proof}
Choose a basis of $\Phi$ such that $\state{+-}$ is represented by its Jordan canonical form. From Proposition \ref{6-8}, under this basis, every $\state{i_1\cdots i_r}$ is represented by a block-diagonal lower-triangular matrix consistently with the block-decomposition of that for $\state{+-}$. Since $M$ is indecomposable, the Jordan canonical form of $\state{+-}$ has to contain only one Jordan block. By Proposition \ref{6-8} again, the matrice of all other $\Pi(T(E))$ elements are then uniquely determined by that of $\state{+-}$. 
\end{proof}

\begin{prop}\label{6-17}
Let $\lambda \in \C$ be generic, i.e., $\lambda \neq p(p-1)K$ for every $p\in \Z_+$. Then there exists an indecomposable module (unique up to isomorphism) of class $\mathscr{C}$ with all composition factors isomorphic to $W(\lambda, [\mu])$. 
\end{prop}

\begin{proof}
Let $\Phi$ be the indecomposable $\Pi(T(E))$-module given in Proposition \ref{6-11}. Let $f_1, ..., f_n$ be a basis such that $\state{+-}$ is represented by the single Jordan block. Let $\Phi_i$ be the submodule of $\Phi$ spanned by $f_i, f_{i+1}, ..., f_n$. We use the simplified notation $W_i^{(0)}$ for $W^{(0)}(\Phi_i, [\mu])$. Then the space $\Phi_i$ forms the lowest weight subspace in $W_i^{(0)}$. We now take quotients of certain submodules similarly as in Section 5. Instead of going through every detail, we just sketch the process with comments on necessary modification:
\begin{enumerate}
    \item If $w$ is an element of $W_i^{(0)}$ such that the lowest weight projection of $Y(v,x)w$ is constantly zero, then $w$ generates a proper submodule. 
    \item The lowest weight projection formula in Proposition \ref{proj0-prop} stays in effect, where the $f$ in (\ref{proj0-formula}) and (\ref{proj0-formula-2}) can be taken as $f_i, f_{i+1}, ..., f_n$. 
    \item The statement of Theorem \ref{5-10} should be modified as follows: Let $j_1, ..., j_r\in \{+, -\}$ satisfying $\sum_{p=1}^r j_p = 0$, $t_1, ..., t_k\in \Z_+$. Set $N =\left|\sum_{p=k+1}^r j_p\right|$. Then there exists a polynomial $P(x)\in \C[x]$ depending only on $j_{k+1}, \cdots j_r$, such that for every $\alpha = i,i+1, ..., n$,
    \begin{align*}
        \state{j_1(-t_1) \cdots j_k(-t_k) j_{k+1}(0) \cdots j_r(0)}f_\alpha -  \state{j_1(-t_1)\cdots j_k(-t_k) j'_{1}(0) \cdots j'_{N}(0)}\left(P(\state{+-})f_\alpha\right)
    \end{align*}
    generates a proper submodule. Here $j_1' = \cdots = j_{N}' = \text{sgn}\left(\sum_{p=k+1}^r j_p\right)$.
    Note that the term $P(\state{+-})f_\alpha$ involves only $f_\beta$ for $\beta\geq \alpha$, and the coefficient of $f_\alpha$ coincides with $C=C_{j_{k+1}...j_r}(\lambda)$ given in Theorem \ref{5-10}. 
    \item Let $W_i^{(1)}$ be the quotient of $W_i^{(0)}$ by elements in (4), we see that $W_i^{(1)}$ has the basis 
    $$\left\{h_{i_1}(-t_1) \cdots h_{i_k}(-t_k)h_{i_{k+1}}(0) \cdots h_{i_r}(0)f_\alpha:
    \begin{aligned}
        &\alpha = i, i+1, ..., n \\
        &r \geq 0, t_1, ..., t_k > 0\\
        &\#\{j:i_j=+\} = \#\{j: i_j=-\}\\
        &i_{k+1} = \cdots = i_r \in \{+, -\}
    \end{aligned}\right\}.$$
    \item The statement of Theorem \ref{5-13} stays in effect, where $f$ in (\ref{submodgen-2}) can be taken as $f_\alpha, \alpha=i, i+1, ..., n$. 
    \item let $W_i$ be the quotient of $W_i^{(1)}$ by all the elements discussed in (5). We see that $W_i$ has the following basis
    $$\left\{h_{+}(-t_1) \cdots h_{+}(-t_{r/2}) h_{-}(-t_{r/2+1}) \cdots h_{-}(-t_r)f_\alpha: 
    \begin{aligned} & \alpha = i, i+1,, ..., n, \\
    & r\geq 0 \text{ even, }t_1\geq \cdots \geq t_{r/2}\geq 1, \\ & t_{r/2+1} \geq \cdots \geq t_{r}\geq 1
    \end{aligned}\right\},$$

    $$\left\{
    h_{+}(-t_1)\cdots h_{+}(-t_{r/2}) h_{-}(-t_{r/2+1}) \cdots h_{-}(-t_k)h_{-}(0)\cdots h_{-}(0)f_\alpha: 
    \begin{aligned}& \alpha = i, i+1,, ..., n, \\
    & r\geq 0 \text{ even, } r/2 \leq k \leq r-1 \\ 
    & t_1\geq \cdots \geq t_{r/2}\geq 1,\\ & t_{r/2+1} \geq \cdots \geq t_{k}\geq 1
    \end{aligned}\right\},$$
    
    $$\left\{
    h_{-}(-t_1)\cdots h_{-}(-t_{r/2}) h_{+}(-t_{r/2+1}) \cdots h_{+}(-t_k)h_{+}(0)\cdots h_{+}(0)f_\alpha: 
    \begin{aligned}& \alpha = i, i+1,, ..., n, \\
    & r\geq 0 \text{ even, } r/2 \leq k \leq r-1 \\ 
    & t_1\geq \cdots \geq t_{r/2}\geq 1,\\ & t_{r/2+1} \geq \cdots \geq t_{k}\geq 1
    \end{aligned}\right\}.$$
\end{enumerate} 
Thus we have constructed the modules
$$W_1\supseteq W_2 \supseteq \cdots \supseteq W_n \supseteq 0.$$
Take $W=W_1$. We show that the lowest weight $V$-module $W$ satisfies the conditions (1) and (2). Condition (1) holds because none of the above quotients have changed the lowest weight subspace. For condition (2), we first note that the quotient $W_i/ W_{i+1}$ is generated by the image $\bar{f_i}$ of $f_i$. Note also that for every polynomial $P(x)\in \C[x]$, $P(\state{+-})f_i$ is a linear combination of $f_i, f_{i+1}, ..., f_n$ with coefficient of $f_i$ being $P(\lambda)$. So the quotient is isomorphic to $W(\lambda,[\mu])$. 

Before arguing the uniqueness, we first make some observations. Fix any $i = 1, ..., n$. Notice if $U, V$ are two submodules of $W_i^{(0)}$ intersecting lowest weight subspace trivially, then so is $U+V$. Thus, there exists a unique maximal submodule of $W_i^{(0)}$ intersecting the lowest weight subspace trivially. One can argue similarly as Theorem \ref{5-18} that the module $W_i$ we constructed above is obtained from taking the quotient of $W_1^{(0)}$ by the unique maximal submodule intersecting the lowest weight subspace trivially. More precisely, Formula (\ref{irred-formula-1}) in the proof of Theorem \ref{5-18} can then be modified as 
\begin{align*}
    \sum a_{j_1 \cdots j_r,\alpha}^{t_1 \cdots t_k} \sum_{1\leq c_k < \cdots < c_1 \leq r} \prod_{p=1}^k(-1)^{m_p-1} &\left(\sum_{\sigma\in S_k} \prod_{p=1}\binom{t_{\sigma(p)}+m_{c_p} -1}{t_{\sigma(p)}}\prod_{p=1}^k t_{\sigma(p)}l\langle i_{c_p}, j_{\sigma(p)}\rangle\right)\nonumber \\
    & \cdot \state{i_1(0)\cdots \widehat{i_{c_k}(0)} \cdots \widehat{i_{c_1}(0)}\cdots i_r(0) j_{k+1}(0)\cdots j_r(0)}f_\alpha = 0.
\end{align*}
Here the summation is also over $\alpha = i, i+1, ..., n$. 
One argues similarly, until arriving at an analogue of Formula (\ref{irred-formula-2}), which in the current situation looks like
\begin{align*}
    \sum_{\alpha=i}^n a_{j_1\cdots j_r, \alpha}^{t_1\cdots t_k} \prod_{p=1}^k (-1)^{m_p-1}\cdot  2^k l^k t_1\cdots t_k\cdot  \state{i_1 \cdots \widehat{i_{c_k}} \cdots \widehat{i_{c_1}} \cdots i_r j_{k+1} \cdots j_r} f_\alpha = 0.
\end{align*}
This is a linear combination of $f_i, ..., f_n$. The coefficient of $f_i$ is precisely a scalar multiple of $f_i$. Since the eigenvalue is generic, the scalar is nonzero. This forces $a_{j_1\cdots j_r, i}^{t_1\cdots t_k} = 0$. Repeating this process for $f_{i+1}, ..., f_n$ to conclude that all $a_{j_1\cdots j_r, \alpha}^{t_1\cdots t_k} = 0$ for every $\alpha = i, i+1, ..., n$. 

Now we show the uniqueness of $W$ using these observations. Let $W$ be any indecomposable lowest weight $V$-module $W$ satisfying (1) and (2) with identical composition factors. We first consider the Jordan canonical form of $\state{+-}$ on the lowest weight space. If the number of Jordan blocks is strictly larger than 1, then from Theorem \ref{4-10}, $W$ can be decomposed by $V$-submodules generated by the generalized eigenvectors for each Jordan block, contradicting our indecomposable assumption. Therefore, the Jordan canonical form of $\state{+-}$ consists of one single Jordan canonical form. Thus $W$ is a quotient of $W_1^{(0)}$ constructed earlier. Since $W$ satisfies Condition (3), the submodule that can appear in the quotient intersects the lowest weight subspace trivially, and should have the same graded dimension as the submodule we used to obtain $W_1^{(2)}$. The conclusion then follows from the uniqueness of the maximal submodule that has trivial intersection to the lowest weight subspace. 
\end{proof}

\begin{nota}
We use the notation $W(\lambda, [\mu], n)$ to denote the lowest weight $V$-module we identified in Proposition \ref{6-17}. 
\end{nota}

On the other hand, the following proposition is certainly surprising. 
\begin{prop}\label{6-18}
Let $\lambda = p(p-1)K$ for some $p\in \Z_+$. Then there does not exists any indecomposable lowest weight $V$-modules satisfying (1) and (2), of length $n\geq 2$. and with identical composition factors $W(\lambda, [\mu])$. 
\end{prop}

\begin{proof}
It suffices the nonexistence for $n=2$. The case for any larger $n$ can be reduced to the $n=2$ case by taking the quotient of $W$ by the submodule $W_3$.  

Suppose the contrary that such a $V$-module $W$ exists. Let $W_{[\mu]}$ be the lowest weight subspace. Let $f_1, f_2$ be a basis of the $W_{[\mu]}$ such that 
$$\state{+-}f_1 = \lambda f_1 + f_2, \state{+-}f_2 = \lambda f_2.$$
Then as a vector space, $W$ has the basis
    $$\left\{h_{+}(-t_1) \cdots h_{+}(-t_{r/2}) h_{-}(-t_{r/2+1}) \cdots h_{-}(-t_r)f_\alpha: 
    \begin{aligned} & \alpha = 1, 2; r\geq 0 \text{ even};\\
    & t_1\geq \cdots \geq t_{r/2}\geq 1, \\ & t_{r/2+1} \geq \cdots \geq t_{r}\geq 1
    \end{aligned}\right\},$$

    $$\left\{
    h_{+}(-t_1)\cdots h_{+}(-t_{r/2}) h_{-}(-t_{r/2+1}) \cdots h_{-}(-t_k)h_{-}(0)\cdots h_{-}(0)f_\alpha: 
    \begin{aligned}& \alpha = 1,2;  r> 0 \text{ even; }\\
    & \max(r-p+1,r/2) \leq k \leq r-1 \\ 
    & t_1\geq \cdots \geq t_{r/2}\geq 1,\\ & t_{r/2+1} \geq \cdots \geq t_{k}\geq 1
    \end{aligned}\right\}.$$
    
    $$\left\{
    h_{-}(-t_1)\cdots h_{-}(-t_{r/2}) h_{+}(-t_{r/2+1}) \cdots h_{+}(-t_k)h_{+}(0)\cdots h_{+}(0)f_\alpha: 
    \begin{aligned}& \alpha = 1,2; r>0 \text{ even; }\\
    & \max(r-p+1,r/2) \leq k \leq r-1 \\ 
    & t_1\geq \cdots \geq t_{r/2}\geq 1,\\ & t_{r/2+1} \geq \cdots \geq t_{k}\geq 1
    \end{aligned}\right\}.$$

    From Theorem \ref{4-10}, the element
    $$w = h_+(-1)^p h_-(0)^p f_1$$
    is in $W$. We now use the lowest weight projection formula to show that $w = 0$. 
    
    First we note that $w$ is a linear combination of basis vectors of weight $p$. Thus $t_1 + \cdots + t_k = p$. This in particular means that $k$ can be at most as large as $p$. 
    
    Second, we note that for every $v =\state{i_1(-m_1)\cdots i_s(-m_s)}\one$, if $s<2p$, then the lowest projection of $Y(v, x)w$ is zero. Indeed, if we write $w = |j_1(-1)\cdots j_p(-1)j_{p+1}(0)\cdots j_{2p}(0)\rangle$ with $j_1 = \cdots = j_p = +, j_{p+1} = \cdots = j_{2p} = -$, then the factor $\prod_{\alpha=1}^p \langle i_{c_\alpha}, j_{\sigma(\alpha)}\rangle$ appearing in each summand in (\ref{proj0-formula}) is always zero if $s< 2p$, as any choice of $p$ elements in $i_1, ..., i_{s}$ has to contain at least one positive. 
    
    With the above two observations, we can use an argument identical to that Proposition \ref{6-17} (cf.  Theorem \ref{5-18}), to see that the coefficients of the basis vector with $k<2p$ are all zero. Since the argument is identical, we should not repeat it here.
    
    Thus, we show that in the expression of $w$ as a linear combination of basis vectors of $W$, all coefficients are zero. Thus $h_+(-1)^p h_-(0)^pf_1 = 0$ in $W$. 
    
    We interpret the vanishing of $h_+(-1)^p h_-(0) f_1$ as follows: if we view the lowest weight subspace $W_{[\mu]}$ as a $\Pi(T(E))$-module and denote it by $\Phi$, then $W$ is the quotient of $W^{(0)}(\Phi, [\mu])$ by some submodule that contains $h_+(-1)^p h_-(0)^p f_1$. However, if we consider the lowest weight projection of 
    $$Y(h_+(-1)^p h_-(-1)^p \one, x) h_+(-1)^p h_-(0)^p f_1,$$
    which, by Theorem \ref{proj0-prop}, is equal to
    $$2^p l^p (p!) h_+(0)\cdots h_+(0)h_-(0)\cdots h_-(0) f_1.$$
    Here, by Proposition \ref{6-8}, Proposition \ref{ConstMult} and the assumption $\lambda = p(p-1)K$, 
    \begin{align*}
        h_+(0)\cdots h_+(0)h_-(0)\cdots h_-(0) f_1 &= \prod_{\beta = 1}^{p} (-h_+\otimes h_-+\beta(\beta-1)K) f_1 \\
        &= \prod_{\beta = 1}^{p-1} (-h_+\otimes h_-+\beta(\beta-1)K) \cdot (-\lambda f_1 +  f_2 + p(p-1)K f_1)\\
        &= \prod_{\beta = 1}^{p-1} (-h_+\otimes h_-+\beta(\beta-1)K)f_2 \\
        &= \prod_{\beta = 1}^{p-1} (-p(p-1)K+\beta(\beta-1)K)f_2.
    \end{align*}
    In other words, the submodule generated by $h_+(-1)^ph_-(0)^p f_1$ contains a nonzero multiple of $f_2$ and thus have nontrivial intersection to the lowest weight subspace. Then $W$, as a quotient of $W^{(0)}(\Phi, [\mu])$, have a one-dimensional lowest weight subspace, contradictory to our assumption. 
\end{proof}

\begin{rema}
Proposition \ref{6-18} implies that for any $V$-module $W$ of class $\mathscr{C}$ with identical composition factors $W(p(p-1)K, [\mu])$ for some $p\in \Z_+$, then the Jordan canonical form of $\state{+-}$ has to be the diagonal matrix $p(p-1)K\cdot I$. It is clear that in this case $W$ is a direct sum of $W(p(p-1)K, [\mu])$.  
\end{rema}

\subsection{$V$-modules of class $\mathscr{C}$: general case} We are now ready for the classification theorem for $V$-modules of class $\mathscr{C}$:

\begin{thm}
Let $W$ be a $V$-module of class $\mathscr{C}$ with lowest weight $\mu\in \C$. Then $W$ is a direct sum of $W(\lambda, [\mu], n)$ with generic $\lambda\in \C$,  and $W(\lambda',[\mu])$ with special $\lambda' \in \{p(p-1)K: p\in \Z_+\}$.  
\end{thm}

\begin{proof}
Take a basis of the lowest weight subspace  $W_{[\mu]}$ that represents $\state{+-}$ in its Jordan canonical form:
$$\text{diag}[J(\lambda_1, n_1), ..., J(\lambda_g, n_g), \lambda_{g+1}, ..., \lambda_{g+s}].$$
Here $\lambda_1, ..., \lambda_g$ are generic, $\lambda_{g+1}, ..., \lambda_{g+s}$ are special. 
Let $\Phi_1, ..., \Phi_g, \Phi_{g+1}, ..., \Phi_{g+s}$ be the corresponding decomposition of $W_{[\mu]}$. For each $i = 1,..., g+s$, let $W_i$ be the $V$-submodule generated by $\Phi_i$. Then for $i = 1, ..., g$, $W_i$ is isomorphic to $W(\lambda_i, [\mu], n_i)$; for $i = g+1, ..., g+s$, $W_i$ is isomorphic to $W(\lambda_i, [\mu])$.  It is also clear from Theorem \ref{4-10} that $W = \bigoplus_{i=1}^{g+s} W_i$. 
\end{proof}

\begin{cor}
If $W$ is a $V$-module of class $\mathscr{C}$ with composition factors being the irreducible $V$-modules generated by eigenfunctions with special eigenvalues, then $W$ is completely reducible. 
\end{cor}

\noindent {\small \sc Pacific Institute of Mathematical Science | University Of Manitoba\\ 451 Machray Hall, 186 Dysart Road, Winnipeg, MB R3T 2N2, Canada}

\noindent {\em E-mail address}: fei.qi@umanitoba.ca

The author states that there is no conflict of interest.


\begin{thebibliography}{KWAK2}

\bibitem[B]{B} R. E. Borcherds, Vertex algebras, Kac-Moody algebras, and the
Monster, \textit{Proc. Natl. Acad. Sci. USA} \textbf{83} (1986), 3068--3071.

\bibitem[BPZ]{BPZ} A. A. Belavin, A. M. Polyakov and A. B. Zamolodchikov, Infinite conformal symmetry
in two-dimensional quantum field theory, \textit{Nucl. Phys.} \textbf{B241} (1984) 333-380.

\bibitem[D]{DoCarmo} M. P. do Carmo, \textit{Differential Geometry of Curves and Surfaces}, Prentice-Hall Inc., New Jersey, 1976.

\bibitem[DLM]{DLM} C. Dong, H.-S. Li and G. Mason, Vertex operator algebras and associative algebras, \textit{J. Algebra}
\textbf{206} (1998), 67-96.


\bibitem[FLM]{FLM} I. Frenkel, J. Lepowsky and A. Meurman, {\it Vertex operator algebra and the monster},  Pure and Appl. Math., 134, Academic Press, New York, 1988.


\bibitem[H1]{H-MOSVA}
Y.-Z. Huang, Meromorphic open string vertex algebras, \textit{J. Math. Phys.} \textbf{54} (2013), 051702. 

\bibitem[H2]{H-MOSVA-Riemann}
Y.-Z. Huang, Meromorphic open-string vertex
algebras and Riemannian manifolds, arXiv:1205.2977.


\bibitem[HK]{HK-OSVA} 
Y.-Z. Huang, L. Kong, Open-string vertex algebras, tensor categories and operads, {\it Comm. Math. Phys.},  \textbf{250} (2004), 433--471.


\bibitem[JS]{JS}
G. Jones, D. Singerman, \textit{Complex functions: an algebraic and geometric viewpoint}, Cambridge University Press, Cambridge, 1987. 


\bibitem[LL]{LL} 
J. Leposwky, H. Li, \textit{Introduction to vertex operator algebras
and their representations}, Progress in Mathematics, 227, Birkh\"auser, Boston, 2004. 


\bibitem[MS]{MS}
G. Moore, N. Seiberg, Classical and quantum conformal field theory, \textit{Comm. Math. Phys.} \textbf{123} (1989), 177--254.

\bibitem[P]{P} P. Petersen,  \textit{Riemannian Geometry}, third edition, Graduate Texts in Mathematics 171, Springer-Verlag, New York (2016).


\bibitem[Q1]{Q-Mod} F. Qi, On modules for meromorphic open-string vertex algebras, \textit{Journal of Mathematical Physics} \textbf{60}, 031701 (2019).

\bibitem[Q2]{Q-Cov-Der} F. Qi, Covariant derivatives of eigenfunctions along parallel tensors over space forms, arxiv:2006.16704.




\end{thebibliography}
\end{document}